\documentclass[twoside,a4paper,12pt,centertags]{amsart}

\usepackage[T1]{fontenc}
\usepackage[utf8]{inputenc}
\usepackage{subfiles}
\usepackage{multicol}

\makeatletter
\usepackage[section]{placeins}

\AtBeginDocument{%
	\expandafter\renewcommand\expandafter\subsection\expandafter{%
		\expandafter\@fb@secFB\subsection
	}%
}
\makeatother
\makeatletter
\@namedef{subjclassname@2020}{%
	\textup{2020} Mathematics Subject Classification}
\makeatother

\usepackage{amsfonts,amssymb,amsmath,amsthm}
\usepackage{mathrsfs}
\usepackage{nicefrac}
\usepackage{latexsym}
\usepackage{mathtools}
\usepackage[noadjust]{cite}
\usepackage{paralist}
\usepackage{longtable}
\usepackage{float}
\usepackage{pdfsync}
\usepackage[dvipsnames]{xcolor}
\usepackage{tikz,tikz-cd}
\usetikzlibrary{patterns}
\usepackage{hhline}
\usepackage{import}
\usepackage{esint}
\usepackage{dsfont}

\usepackage{enumitem}

\usepackage[colorlinks=true,citecolor=magenta,linkcolor=cyan,backref]{hyperref}
\AtBeginDocument{}

\usepackage[
  a4paper,
  textwidth=6in,
  textheight=9in,
  centering
]{geometry}

\numberwithin{equation}{section}

\theoremstyle{plain}
\newtheorem{thm}{Theorem}[section]
\newtheorem{lem}[thm]{Lemma}

\newtheorem{prop}[thm]{Proposition}

\theoremstyle{definition}
\newtheorem{defn}[thm]{Definition}
\newtheorem{rem}[thm]{Remark}

\usepackage{color}

\DeclareMathOperator*{\esssup}{ess\,sup}

\newcommand{\IC}{\mathbb{C}}

\newcommand{\R}{\mathbb{R}}

\newcommand{\cA}{\mathcal{A}}

\newcommand{\cD}{\mathcal{D}} 

\newcommand{\cJ}{\mathcal{J}}
\newcommand{\cK}{\mathcal{K}} 
\newcommand{\cL}{\mathcal{L}}

\renewcommand{\L}{\operatorname{L}} 
\newcommand{\C}{\operatorname{C}} 
\renewcommand{\H}{\operatorname{H}} 
\newcommand{\W}{\operatorname{W}}
\DeclareRobustCommand{\Hdot}{\dot{\H}\protect{\vphantom{H}}
} 
 
\renewcommand\Re{\operatorname{Re}}

\def\angle#1#2{\langle #1,#2 \rangle} 

\def\tpsi{\tilde\psi}

\def\v{\phi}

\newcommand{\norm}[1]{\| {#1}\|}
\newcommand{\HZK}{\mathcal K}

\newcommand{\dd}{\, \mathrm{d} }

\title[Fundamental solutions to non-local kinetic equations]{Existence and upper bounds of fundamental solutions to kinetic equations with rough non-local diffusion}
\author{Pascal Auscher}
\address[Pascal Auscher] {Universit\'e Paris-Saclay, CNRS, Laboratoire de Math\'{e}matiques d'Orsay, 91405 Orsay, France}
\email{pascal.auscher@universite-paris-saclay.fr}
\author{Lukas Niebel}
\address[Lukas Niebel]{ETH Z\"urich, Department of Mathematics, R\"amistrasse 101, 8092 Z\"urich, Switzerland.}
\email{lukas.niebel@math.ethz.ch}
\date{September 30, 2026}



\begin{document}

\begin{abstract}
  We study the fundamental solution operators associated to Kolmogorov equations with symmetric non-local diffusion with rough kernels. We prove that these operators are induced by measurable kernels and prove sharp upper bounds in terms of a kinetic profile governed by the interplay of transport and non-local diffusion. Our proof combines the small-jump with loss decomposition of P.-A.~Meyer with the ideas of Nash. The proof is purely analytic.
\end{abstract}

\maketitle

\tableofcontents

\section{Introduction}

Fix $d\ge1$ and $0<\beta<1$.
We work on a finite time interval $[0,T]$. We consider the Kolmogorov equation with symmetric non-local diffusion operator with rough kernel
\begin{equation}
  \label{eq:main-equation}
  (\partial_t+v\cdot\nabla_x)f+\cA_k f=0,
  \qquad (t,x,v)\in(0,T)\times\R^d\times\R^d,
\end{equation}
with a suitable initial condition, where
\begin{align}
  \label{eq:form}
  \cA_k f(t,x,v)= 2\ \text{p.v.} \int_{\R^d}
  (f(t,x,v)-f(t,x,v'))
  k(t,x,v,v')\dd v' .
\end{align}

The kernel $k$ is assumed to be measurable, and to satisfy for some $0<\lambda_k\le \Lambda_k<\infty$
\begin{align}
  \label{eq:stable-comparable}
  \lambda_k |v-v'|^{-d-2\beta}
   & \le k(t,x,v,v')=k(t,x,v',v)
  \le \Lambda_k |v-v'|^{-d-2\beta}
\end{align}
for almost every $(t,x,v,v') \in [0,T] \times \R^{3d}$.
The corresponding functional analytic setup will be made precise below.

Our main result is the following.

\begin{thm}[Existence and decay of the fundamental solution]\label{thm:main}
  Assume \eqref{eq:stable-comparable} and consider
  \eqref{eq:main-equation} with $\cA_k$ given by \eqref{eq:form}.
  The fundamental solution operators $\Gamma(t,s)$ exist for all $0\le s<t\le T$ and are induced by a nonnegative measurable kernel with the following estimates,
  writing $\tau=t-s>0$,
  \begin{equation}\label{eq:main}
    0\le {} \Gamma(t,x,v,s,y,w)
    \le {}  \frac C{\tau^{d+d/\beta}}
    \ \HZK_{d+2\beta}\!\left(
    \frac{x-y-\tau w}{\tau^{1+1/(2\beta)}},
    \frac{v-w}{\tau^{1/(2\beta)}}
    \right)
  \end{equation}
  for almost every $(x,v,y,w) \in \R^{4d}$, where
  \begin{equation}
    \label{cmp:eq:K-def}
    \HZK_{d+2\beta}(x,v)
    =(1+|x|+|v|)^{-d-2\beta}
    \int_0^1(1+|x-sv|)^{-d-2\beta}\dd s.
  \end{equation}
  The constant $C$ depends only on $d,\beta, \lambda_{k}, \Lambda_{k}$.
\end{thm}

\begin{rem}
  Further properties of the fundamental solution operators and their kernels can be found in Theorem \ref{thm:representation0T} and Theorem \ref{thm:nash-kernel}.
\end{rem}

\begin{rem}
  We call the profile in \eqref{cmp:eq:K-def} the \emph{kinetic profile of order $2\beta$} as the decay from the integral term exhibits  a  combination of the action of the diffusion of order $2\beta$ and the kinetic transport term. In the homogeneous case, that is for solutions independent of $x$ and there is no transport, we recover the typical polynomial decay for non-local diffusion of order $2\beta$. This kinetic profile appears in the sharp heat-kernel estimates for kinetic equations with non-local diffusion with kernels of the form $\kappa(v')|v-v'|^{-d-2\beta}$ with $0<\lambda \le \kappa \le \Lambda$, such as the fractional Kolmogorov equation
  \[
    (\partial_t+v\cdot\nabla_x)f + (-\Delta_v)^{\beta}f = 0,
  \]
  in \cite{Hou2026}. There sharp bounds on the fundamental solution from above and from below in terms of the kinetic profile \eqref{cmp:eq:K-def} are proven. We extend the upper bound to the rough kernel setting. In particular, we provide an alternative proof of their result. 
\end{rem}

The study of Kolmogorov equations with non-local rough diffusion is one of the central research tasks in kinetic regularity theory. This is due to its relation to the Boltzmann equation \cite{MR4433077}. A landmark result was the proof of the weak Harnack inequality in \cite{MR4049224}; see also \cite{MR4688651}. The strong Harnack inequality was proven to be false in general \cite{MR4828346}. This is related to the fact that a Moser-type $\L^2-\L^\infty$ cannot hold without a tail; see also \cite{anceschi2026giorginashmosertheorykineticequations,loher2026semilocalbehaviournonlocalhypoelliptic}. A Harnack inequality for nonnegative solutions localised in $(t,x)$ and $v\in\R^d$ was proven in \cite{loher2026semilocalbehaviournonlocalhypoelliptic}. Weak solutions to \eqref{eq:main-equation} were constructed in \cite{AIN} and the corresponding fundamental solution operators were introduced in \cite{AIN2}. The precise kernel estimates even in the constant coeffficient case remained an open problem for a long time. A preliminary result was obtained in \cite{loher2026semilocalbehaviournonlocalhypoelliptic}. The sharp kernel estimates in the constant coefficient case were  obtained only recently in \cite{Hou2026}. We also mention the one-dimensional result \cite{grube2024pointwiseestimatesfundamentalsolution}.

Let us now summarise our strategy.
\begin{enumerate}[label=(\arabic*)]
  \item We recall the existence and uniqueness of fundamental solution operators for kinetic equations with non-local diffusion and rough kernels and their backward equations; compare \cite{AIN2}.
  \item We introduce the class of equations we consider: kernels may be restricted to small jumps and we allow for potentials.
  \item We prove contractivity and positivity of the fundamental solution operators on $\L^p$ spaces, $1\le p\le\infty$.
  \item We establish hypercontractivity of the fundamental solution operators through an anisotropic Nash inequality. We deduce the existence of a measurable kernel of the fundamental solution operator and obtain on-diagonal upper bounds.
  \item We implement Davies' method for non-local kinetic equations on $\L^p$ spaces, $2\le p<\infty$ following the ideas of Carlen--Kusuoka--Stroock developed in parabolic context \cite{MR898496}.
  \item The previous step introduces an exponential weight parametrised via a Lipschitz function. Now, we obtain precise decay estimates in terms of this Lipschitz function using complex interpolation and iteration. Here, we focus on kernels that are restricted to small jumps.
  \item We optimise the Lipschitz function to obtain off-diagonal rapid decay for the fundamental solution of equations with kernels that have small-jumps only.
  \item In the previous steps we considered mostly zero potentials. By a comparison argument, the operator and kernel upper bounds extend to bounded nonnegative potentials.
  \item At the heart of our proof is a small-jump decomposition with loss potential \`a la P.-A.~Meyer \cite{MR415784} (which was also a key ingredient in \cite{Hou2026}). We assemble all of the estimates and complete the proof.
\end{enumerate}

Many of these steps may seem familiar to readers working on similar problems in a parabolic context and were available already in the literature. Step (1) is interesting in relation to existence. It furnishes a universal object, the family of fundamental solution operators, with which one can work whether or not they are represented by  a kernel, usually called a fundamental solution in the literature.  This universality in parabolic context was emphasised in \cite{MR4924261}. Working without a priori access to a kernel was the key point of the construction in \cite{AIN2} in the kinetic context, because existence of a kernel requires further assumptions like the ones we use here. Concerning kernel estimates, 
one key insight is that \cite{Hou2026} implements a Meyer decomposition in the kinetic setting for constant kernels. The importance of the Meyer decomposition for rough kernels was mentioned in \cite{MR4530315,loher2026semilocalbehaviournonlocalhypoelliptic}. Another key part of our strategy is a Moser iteration at the level of mapping properties of the fundamental solution operators, which is close to the original ideas of Nash \cite{MR100158}; see also \cite{MR855753,MR898496}.
This is combined with the Davies method. We emphasise that the strategy in the present work differs from the one  for local diffusion in \cite{AIN2} that relied on Moser-type $\L^2-\L^\infty$ bounds instead of the iteration.  It is likely though that a different proof can be obtained using the strategy here.

In the case of kinetic equations with local diffusion, upper and lower estimates for the fundamental solution operators, assuming the existence of the kernel, were provided in \cite{MR4181953,MR3717350}. The fundamental solution operators in this case were constructed in \cite{AIN2}, where, using  the Davies method  in kinetic context for the first time,  it was also proved that they are induced by measurable kernels satisfying the sharp upper bound. Subsequently,  \cite{dietert2025nashsgboundkolmogorov} obtains a proof of the lower bound following the ideas of Nash \cite{MR100158}; see also \cite{dietert2025criticaltrajectorieskineticgeometry}.

In the non-local parabolic setting, which corresponds to the homogeneous kinetic setting where functions are constant in $x$, our argument provides an alternative proof of the existence of the fundamental solution operator and its kernel, and of its upper estimates. We refer to \cite{MR3096796,MR5123698,MR4530315,MR1895210,MR2008600} and the references therein.

\subsection*{Acknowledgments}

The first author's visit to M\"unster in June 2026 was supported by the University
of M\"unster.
Part of this work was done during the workshop ``Regularity Theory for Evolution Equations'' at the SwissMAP Research Station, Les
Diablerets. The authors thank the organisers and the SwissMAP Research Station for their hospitality.

Lukas Niebel is funded by SNSF Starting Grant
TMSGI2\textunderscore226018 and by the Deutsche Forschungsgemeinschaft (DFG,
German Research Foundation) under Germany's Excellence Strategy
EXC 2044/2--390685587, Mathematics M\"unster:
Dynamics--Geometry--Structure.

A CC-BY 4.0 \url{https://creativecommons.org/licenses/by/4.0/} public copyright license has been applied by the authors to the present document and will be applied to all subsequent versions up to the Author Accepted Manuscript arising from this submission.

\subsection*{Declaration of AI use.}
The mathematical work presented in this paper was carried out by the authors without the use of AI tools.

\section{$\L^2$ theory for the fundamental solution operator}

\subsection{General form setting}

  We recall the form setting and the $\L^2$ theory for weak solutions from \cite{AIN}, using the presentation of the homogeneous spaces in \cite{AN}. We then recall the construction of fundamental solution operators from \cite{AIN2}.

Throughout, we assume $0<\beta<1$. This setup also works for complex forms.
We define for $f$ locally square integrable,
\[\|f\|_{\Hdot^\beta_{v}}^2 = \int_{{\R}^{2d}} \frac{|f(v)-f(v')|^2}{|v-v'|^{d+2\beta}}\dd (v,v'),
\]
and $\H^\beta_{v}=\{f\in \L^2_{v}\, ; \, \|f\|_{\Hdot^\beta_{v}}<\infty\}$ with
\[\|f\|_{\H^\beta_{v}}^2= \|f\|_{\Hdot^\beta_{v}}^2+ \|f\|_{\L^2_{v}}^2.
\]
It is classical that this space agrees with the potential space $(I-\Delta_{v})^{-\beta/2}\L^2_{v}$ and that $ \|(I-\Delta_{v})^{\beta/2} f\|_{\L^2_{v}}$ is an equivalent norm.

We fix $T>0$.
Consider a sesquilinear form
$$a_{t,x}\colon \H^{\beta}_{v} \times \H^\beta_{\vphantom{v} v} \to \IC$$ with
\begin{equation}
  \label{eq:ellip-upper}
  |a_{t,x}(f,g)|\le \Lambda \|f\|_{\Hdot^\beta_{v}} \|g\|_{\Hdot^\beta_{v}}+c^0\|f\|_{\L^2_{v}}\|g\|_{\L^2_{v}}
\end{equation}
for almost every $(t,x)\in [0,T]\times \, \R^d$, with $\Lambda>0$ and $c^0 \ge 0$. We assume that for $f,g \in \L^2_{t,x}\H^{\beta}_{\vphantom{t,x} v}$, the function
$(t,x)\mapsto a_{t,x}(f,g)$ is measurable and we set
\begin{equation}
  \label{eq:a-defi}
  a(f,g)=\int_{(0,T)\times \R^d} a_{t,x}(f,g)\dd (t,x).
\end{equation}
For short, we have written  $a_{t,x}(f(t,x, \cdot),g(t,x,\cdot))$ as $a_{t,x}(f,g)$ in \eqref{eq:a-defi}. In the notation $ \L^2_{t,x}\H^{\beta}_{\vphantom{t,x} v}$, we do not indicate the time interval when it is clear from context. From \eqref{eq:ellip-upper}, we deduce that $a$ is a continuous sesquilinear form on $\L^2_{t,x}\H^{\beta}_{\vphantom{t,x} v}$.
As $\L^2_{t,x}\H^{-\beta}_{\vphantom{t,x} v}$ identifies with the dual of $\L^2_{t,x}\H^{\beta}_{\vphantom{t,x} v}$ in the duality extending the $\L^2_{t,x,v}$ inner product, we can set $\cA: \L^2_{t,x}\H^{\beta}_{\vphantom{t,x} v} \to \L^2_{t,x}\H^{-\beta}_{\vphantom{t,x} v}$ the bounded operator associated to $a$ by
$$ \angle{\cA f}{g}= a(f,g), \quad f,g\in \L^2_{t,x}\H^{\beta}_{\vphantom{t,x} v}.$$

We shall work with coercivity in the sense that  there exists $\lambda>0$, $ c_{0}\in \R $ such that for all $(t,x)\in (0,T)\times \, \R^d$ and all $f\in \H^\beta_{v}$,
\begin{align}
  \label{eq:weakelliptic-bis}
  \Re a_{t,x}(f,f) \ge \lambda \|f\|_{\Hdot^\beta_{v}}^2 -c_{0}\|f\|^2_{\L^2_{v}}.
\end{align}

\subsection{Weak solutions and fundamental solution operators}

Let $\cA$ be as above with \eqref{eq:ellip-upper} and \eqref{eq:weakelliptic-bis}.

\begin{defn}[Weak solutions to the kinetic Cauchy problem]
  Let $0<T<\infty$.
  Let $S \in \cD' ((0,T)\times \R^d\times \R^d)$ and $\psi \in \L^2_{x,v}$. Assume that the distribution $S$ extends to a continuous linear functional on $\cD ([0,T)\times \R^{d}\times \R^d)$. A distribution $f \in \cD' ((0,T)\times \R^{d}\times\R^d)$ is said to be a \emph{weak solution} {to the Cauchy problem on $[0,T]$} \begin{align}
    \label{eq:CP}
    \begin{cases}
      (\partial_{t}+v\cdot\nabla_{x})f + \cA f= S, & \\
      f(0)=\psi,                                   &
    \end{cases}
  \end{align}
  if $f\in \L^2_{t,x}\H^{\beta}_{\vphantom{t,x} v}$ and for all $h\in \cD ([0,T)\times \R^{d}\times\R^d)$,
  \[
    -\int_{(0,T)\times\R^{2d}} f \, (\partial_{t}+v\cdot\nabla_{x}) \overline h \dd (t,x,v) + a(f,h) = \angle{S}{h} + \int_{\R^{2d}} \psi (x,v) \overline h (0,x,v) \dd (x,v).
  \]
\end{defn}

One can also define weak solutions of the final value problem for the adjoint backward equation $-(\partial_{t}+v\cdot\nabla_{x})\tilde f + \cA^* \tilde f= \tilde S$, $\tilde f(T)=\tilde \psi$.

  We recall a special case of the well-posedness result from \cite{AIN}, together with the energy identity from \cite[Lemma~4.2]{AIN}.
  \begin{thm} \label{thm:L2-wellposedness}
    For $S\in\L^1((0,T);\L^2_{x,v})$ and $\psi\in\L^2_{x,v}$, the Cauchy problem \eqref{eq:CP} has a unique weak solution. It belongs to $\C([0,T];\L^2_{x,v})$, the function $t\mapsto\|f(t)\|_{\L^2_{x,v}}^2$ is absolutely continuous, and
    \begin{equation}
      \label{eq:L2-energy}
      \frac12\|f(t)\|_{\L^2_{x,v}}^2
      +\Re\int_s^t\int_{\R^d}a_{\tau,x}(f,f)\dd x\dd\tau
      =\frac12\|f(s)\|_{\L^2_{x,v}}^2
      +\Re\int_s^t\angle{S(\tau)}{f(\tau)}\dd\tau
    \end{equation}
    for $0\le s\le t\le T$, the brackets $\angle{\ }{\ }$ denoting the $\L^2_{x,v}$ inner product. The analogous assertion holds for the adjoint backward final value problem.
  \end{thm}

We now recall the definition of a fundamental solution operator.
We use the tensor product notation $\v \otimes \psi$ for products of functions $\v$ of the $t$ variable and functions $\psi$ of the $(x,v)$ variables, that is $(\v \otimes \psi)(t,x,v)=\v(t)\psi(x,v)$. By a tensor product of test functions $\v \otimes \psi$, we mean $\v\in \cD(0,T)$ and $\psi\in \cD(\R^{2d})$.  

\begin{defn}[Definition of a fundamental solution]
  A fundamental solution to $(\partial_{t}+v\cdot\nabla_{x}) + \cA$ is a family of
  operators $\Gamma=(\Gamma(t,s))_{ (s, t)\in [0,T]^2}$, called fundamental solution operators, with the following properties.
  \begin{enumerate}
    \item They are uniformly bounded operators on $\L^2_{x,v}$.
    \item $\Gamma(t,s)=0$ for almost every $(s,t)$ with $s>t$.
    \item For any $\psi,\tpsi \in \cD(\R^{2d})$, $(s,t)\mapsto \angle {\Gamma(t,s)\psi}{\tpsi}$ is a locally integrable function:
          \[\int_{K\times\tilde K} |\angle {\Gamma(t,s)\psi}{\tpsi}|\dd (s,t) <\infty
          \]
          for compact subsets $K, \tilde K$ in $(0,T)$.
    \item For any tensor product of test functions $S=\phi\otimes \psi$, the weak solution $f$ to the kinetic Cauchy problem
          $(\partial_{t}+v\cdot\nabla_{x})f + \cA f=S$ on $(0,T)$ with $f(0)=0$, satisfies
          \begin{align}
            \angle{f(t)}{\tpsi}=\int_{(0,t)} \phi(s)\angle {\Gamma(t,s)\psi}{\tpsi}\, \dd s
          \end{align}
          for all $\tpsi\in \cD(\R^{2d})$ and almost every $t\in (0,T]$.
  \end{enumerate}

  One defines the fundamental solution $\widetilde\Gamma=(\widetilde\Gamma(s,t))_{ (s,t) \in [0,T]^2}$ to the backward operator  $-(\partial_{s}+v\cdot \nabla_{x})+\cA^*$ analogously and (ii) is replaced by $ \widetilde\Gamma(s,t)=0$ if $s>t$.
\end{defn}

  The fundamental solution operators exist and are unique by \cite[Theorems~3.3 and~3.6]{AIN2}. We collect the representation formula and the operator identities that we shall use below.

  \begin{thm}[Representation of weak solutions]
    \label{thm:representation0T}
    Let $0\le s<T$, $\psi\in\L^2_{x,v}$ and $S\in\L^1((s,T);\L^2_{x,v})$. Then the unique weak solution with initial value $\psi$ at time $s$ is given by
    \begin{equation}
      \label{eq:L2-duhamel}
      f(t)=\Gamma(t,s)\psi+\int_s^t\Gamma(t,r)S(r)\dd r,
      \qquad s\le t\le T,
    \end{equation}
    in $\L^2_{x,v}$. The integral is a Bochner integral in $\L^2_{x,v}$. The analogous representation holds for the adjoint backward final value problem using $\widetilde\Gamma$.

    Moreover, the fundamental solution operators satisfy the following properties.
    \begin{enumerate}
      \item For every $s\in[0,T]$, the map $[s,T]\ni r\mapsto\Gamma(r,s)$ is strongly continuous on $\L^2_{x,v}$. For every $t\in[0,T]$, the map $[0,t]\ni r\mapsto\widetilde\Gamma(r,t)$ is strongly continuous on $\L^2_{x,v}$.
      \item We have the adjoint relation $\widetilde\Gamma(s,t)=\Gamma(t,s)^*$ for $0\le s\le t\le T$.
      \item We have the Chapman--Kolmogorov relation $\Gamma(t,s)=\Gamma(t,r)\Gamma(r,s)$ for $0\le s<r<t\le T$.
    \end{enumerate}
  \end{thm}

\subsection{Energy estimates and transfer of regularity}

  We now recall the transfer-of-regularity estimates for $\cA$ satisfying \eqref{eq:ellip-upper} and \eqref{eq:weakelliptic-bis}.

  We set
  \[
    D_v=(-\Delta_v)^{1/2},\qquad D_x=(-\Delta_x)^{1/2},\qquad
    \gamma_\beta=\frac{\beta}{2\beta+1}.
  \]

  \begin{prop}[Transfer of regularity] \label{prop:transferofregularity}
    Let $0\le s<T$, $\psi\in\L^2_{x,v}$, $S\in\L^1(s,T;\L^2_{x,v})$ and $f$ be given by \eqref{eq:L2-duhamel}. For $s\le t_0<t_1\le T$, we write 
    $\L^1_t\L^2_{x,v}=\L^1(t_0,t_1; \L^2(\R^{2d}))$, $\L^2_{t,x,v}=\L^2((t_0,t_1)\times\R^{2d})$. Then
    \begin{align}
      \label{eq:forward-transfert}
      \sup_{t\in[t_0,t_1]}\|f(t)\|_{\L^2_{x,v}}
      +\|D_{\vphantom{t}v}^\beta f\|_{\L^2_{t,x,v}}
      +\|D_{\vphantom{t}x}^{\gamma_\beta}f\|_{\L^2_{t,x,v}}
      \le C\left(\|f(t_0)\|_{\L^2_{x,v}}+\|S\|_{\L^1_t\L^2_{x,v}}\right),
    \end{align}
    where $C$ depends on $d,\beta,\lambda,\Lambda,c_0$ and $t_1-t_0$. For $0<t_1-t_0\le1$, the constant depends only on $d,\beta,\lambda,\Lambda,c_0$. The same holds for the adjoint equation.
  \end{prop}

  \begin{proof}
    The forward estimate follows from \cite[Theorems~5.2 and~6.3]{AIN}. The backward estimate follows by the same argument, as explained in \cite[Remark~5.4]{AIN}.
  \end{proof}

\subsection{Our special class of forms and small-jump forms with loss potential}

At this stage, the $\L^2_{x,v}$ theory is valid for very general forms even with complex coefficients provided  \eqref{eq:ellip-upper} and \eqref{eq:weakelliptic-bis} hold.
We introduce the special class of forms used in this work and to which our theory applies.

 For kernels $k$ and potentials $V$ satisfying the assumptions below, we set
  \begin{equation} \label{eq:kernel-potential-form}
    a^{k,V}_{t,x}(f,g)
    =\int_{\R^{2d}}(f-f')(\bar g-\bar g')k\dd (v,v')
    +\int_{\R^d} V f\bar g\dd v,
  \end{equation}
  where we abbreviate $f=f(v)$, $f'=f(v')$, $k=k(t,x,v,v')$, and $V=V(t,x,v)$.
  The form $a^{k,V}$ is defined as in \eqref{eq:a-defi}.
  We introduce the associated operators by $\cA_{k,V}$ and put
  $\cA_k=\cA_{k,0}$.
  We write $\Gamma_{k,V}$ and $\widetilde\Gamma_{k,V}$ for the associated forward and backward fundamental solution operators. Their existence is part of Lemma~\ref{lem:basic-properties} below.

  Equation \eqref{eq:main-equation} with a real-valued symmetric kernel $k$ satisfying \eqref{eq:stable-comparable} can be rewritten  as $\cA=\cA_{k,0}$ with $a=a^{k,0}$ given by
  \eqref{eq:kernel-potential-form}. The form bounds hold with $\lambda=\lambda_k$, $\Lambda=\Lambda_k$ and $c^0=c_0=0$. We are interested in real-valued functions, but we retain the complex formulation for the interpolation argument below. We remark that our arguments below   for the bilinear form $\cA_{k,0}$ go through without the symmetry of $k$ 
  with slight modifications.  
  We keep assuming it for simplicity.

We introduce the small-jump/large-jump decomposition of P.-A.~Meyer \cite{MR415784}.
We adopt the standard terminology of decomposing a non-local operator such as the Boltzmann collision operator into a gain and loss term.
Define
\begin{equation}
  k_{\le1}=k\mathds{1}_{\{|v-v'|\le1\}},
  \qquad
  k_{>1}=k\mathds{1}_{\{|v-v'|>1\}}.
\end{equation}
The small-jump form is $a^{k_{\le1},0}$. We consider the family of large-jump gain operators $J_{t,x}:\L^2_{v}\mapsto \L^2_{v}$ defined by
\begin{align*}
  \langle J_{t,x}f, g\rangle = \int_{\R^d\times \R^d}k_{>1} (f \bar g'+ f' \bar g) \dd (v,v') =2 \int_{\R^d\times \R^d}k_{>1} f \bar g' \dd (v,v')
\end{align*}
and the loss potential
\begin{equation}
  \ell(t,x,v)=2\int_{|h|>1}k(t,x,v,v+h)\dd h = 2\int_{\R^d}k_{>1}(t,x,v,v')\dd v'.
\end{equation}
We consider the operator $\cA_k+\cJ$, where
\[
  \langle \cJ f, g\rangle = \int_{(0,T)\times \R^d} \langle J_{t,x}f, g\rangle\dd (t,x),
\]
using the convention of ignoring $t,x$ in the notation for $f,g$. We call $\cA_{k_{\le1},\ell}$ the small-jump operator with loss potential.

  For the fixed kernel satisfying \eqref{eq:stable-comparable}, we set
  \begin{equation}
    \Gamma=\Gamma_{k,0},\qquad \Sigma=\Gamma_{k_{\le1},\ell}.
  \end{equation}
  The notation with no indices always refers to the fixed kernel of \eqref{eq:stable-comparable}.

  For the estimates below, we allow any measurable, symmetric, nonnegative kernel $k$ satisfying
  \begin{equation}
    \label{eq:upperbound}
    k(t,x,v,v')\le\Lambda_k|v-v'|^{-d-2\beta}
  \end{equation}
  and
  \begin{equation} \label{eq:local-stable-lower}
    k(t,x,v,v')\ge\lambda_k
    \frac{\mathds{1}_{\{|v-v'|\le1\}}}{|v-v'|^{d+2\beta}}
  \end{equation}
  almost everywhere, for some $0<\lambda_k\le\Lambda_k$.
  We also assume that the potential $V\in\L^\infty((0,T)\times\R^{2d})$ is nonnegative.

  \begin{lem}[Basic properties]
    \label{lem:basic-properties}
    Let $k$ be measurable, symmetric and nonnegative and satisfy \eqref{eq:upperbound} and \eqref{eq:local-stable-lower}, and let $V\in\L^\infty$ be nonnegative. We put
    \[
      M_1=\int_{|h|>1}|h|^{-d-2\beta}\dd h.
    \]
    We have the following properties.
    \begin{enumerate}
      \item $\ell$ is bounded with $0\le\ell\le\ell_1:=2\Lambda_kM_1$.
      \item $\cJ$ is a bounded operator on $\L^2_{t,x,v}$ with bound $\ell_1$.
      \item $\cA_k+\cJ=\cA_{k_{\le1},\ell}=\cA_{k_{\le1},0}+\ell$.
      \item The inequalities \eqref{eq:ellip-upper} and \eqref{eq:weakelliptic-bis} hold for $a^{k,V}$ with $\lambda=\lambda_k$, $\Lambda=\Lambda_k$, $c_0=4\lambda_kM_1$ and $c^0=\|V\|_\infty$. More precisely,
              \begin{equation}
                \label{eq:local-kernel-coercivity}
                a^{k,V}_{t,x}(f,f)\ge a^{k,0}_{t,x}(f,f)
                \ge\lambda_k\|f\|_{\Hdot^\beta_v}^2
                -4\lambda_kM_1\|f\|_{\L^2_v}^2.
              \end{equation}
              The same bounds holds with $k_{\le1}$ in place of $k$. In particular, they hold for $a^{k_{\le1},0}$ with $c^0=0$ and for $a^{k_{\le1},\ell}$ with $c^0=\ell_1$.
      \item The fundamental solution operators $\Gamma_{k,V}(t,s)$ and $\widetilde\Gamma_{k,V}(s,t)$ exist and are real operators.
      \item $\Gamma_{k,V}(t,s)$ and $\widetilde\Gamma_{k,V}(s,t)$ are contractions on $\L^2_{x,v}$.
    \end{enumerate}
  \end{lem}

\begin{proof}
  By \eqref{eq:upperbound}, the positivity of $\ell$ is obvious and
  \begin{align}
    \ell(t,x,v)
     & \le \ell_1
    :=2\Lambda_k\int_{|h|>1}|h|^{-d-2\beta}\dd h<\infty.
  \end{align}
  Next, it is clear that $J_{t,x}$ are bounded on $\L^1_{v}$ and $\L^\infty_{v}$ with norm bounded by $\ell_{1}$. Hence, interpolation and integration in $(t,x)$ proves the second item.
  For the third item, we have the identity
  \begin{align*}
     & \int_{(0,T)\times\R^{3d}} (f - f')(\bar g-\bar g') k \dd (t,x,v,v') +\int_{(0,T)\times\R^{3d}} k_{>1} ( f \bar g'+f'\bar g) \dd (t,x,v,v')
    \\
     & = \int_{(0,T)\times\R^{3d}} (f - f')(\bar g-\bar g') k_{\le1} \dd (t,x,v,v') +\int_{(0,T)\times\R^{3d}} k_{>1}( f \bar g+ f'\bar g')\dd (t,x,v,v').
  \end{align*}

The upper bound for $a^{k,V}$ follows from \eqref{eq:upperbound} and $0\le V\le\|V\|_\infty$. For the lower bound, we fix $(t,x)$ and estimate
    \[
      \int_{|v-v'|>1}\frac{|f(v)-f(v')|^2}{|v-v'|^{d+2\beta}}\dd (v,v')
      \le4M_1\|f\|_{\L^2_v}^2.
    \]
    Then,
    \begin{align*}
      a^{k,V}_{t,x}(f,f)\ge a^{k,0}_{t,x}(f,f)
       & \ge \lambda_k\int_{|v-v'|\le1}
      \frac{|f(v)-f(v')|^2}{|v-v'|^{d+2\beta}}\dd(v,v')
      \\
       & =\lambda_k\|f\|_{\Hdot^\beta_v}^2
      -\lambda_k\int_{|v-v'|>1}
      \frac{|f(v)-f(v')|^2}{|v-v'|^{d+2\beta}}\dd(v,v')
      \\
       & \ge \lambda_k\|f\|_{\Hdot^\beta_v}^2
      -4\lambda_kM_1\|f\|_{\L^2_v}^2,
    \end{align*}
    which proves \eqref{eq:local-kernel-coercivity}. The kernel $k_{\le1}$ satisfies \eqref{eq:upperbound} and \eqref{eq:local-stable-lower} with the same constants. Taking $V=\ell$ and using $\|\ell\|_\infty\le\ell_1$ gives the claims in the fourth assertion.

    The form bounds and the abstract theory give the existence of $\Gamma_{k,V}(t,s)$ and $\widetilde\Gamma_{k,V}(s,t)$. These operators are real by uniqueness. Applying Theorem \ref{thm:L2-wellposedness} and its backward analogue with zero source, we deduce $\L^2$-contractivity from the non-negativity of the forms.
  \end{proof}

\begin{rem}
  The work \cite{MR4049224} develops a theory from \eqref{eq:form} under more relaxed assumptions on $k$ motivated by the Boltzmann equation.  Our work treats  what they call the symmetric part but under stricter assumptions on the kernel. We have not attempted to prove the same bounds for their anti-symmetric part, for which we could not rely on positivity. 
\end{rem}

\section{$\L^1$ bounds, kernels, on-diagonal bounds and interpolation}

  In this section, we extend $\Gamma_{k,V}(t,s)$ to positive contractions on $\L^p_{x,v}$. We then consider $V=0$ and use the Nash method to obtain kernel representations and on-diagonal bounds.

\subsection{Extension to $\L^p$ spaces}

  Throughout this section, $k$ is measurable, symmetric and nonnegative, and satisfies \eqref{eq:upperbound} and \eqref{eq:local-stable-lower}. In this subsection, we also allow a nonnegative potential $V\in\L^\infty((0,T)\times\R^{2d})$ in form \eqref{eq:kernel-potential-form}.

\begin{lem}[Lipschitz preservation of Sobolev spaces] \label{lem:L1-lipschitz-calculus}
  Let $\chi: \R \to \R $ be Lipschitz with constant $\operatorname{Lip}(\chi)$ and satisfy $\chi(0) = 0$. Then for real-valued functions $f$
  \[
    f \in \H^\beta_v \quad \implies \quad \chi(f) \in \H^\beta_v,
    \qquad
    \| \chi(f) \|_{\H^\beta_v}\le \operatorname{Lip}(\chi) \|f\|_{\H^\beta_v}.
  \]
  In particular, composition with Lipschitz functions is continuous on $\H^\beta_v$. The same assertions hold in $\L^2_x\H^\beta_v$ after integration in $x$.
\end{lem}

\begin{proof}
  The result can be found in \cite[Sec.~2, (2.1)--(2.2)]{MR2944369}.
\end{proof}

  \begin{lem}[Positivity preservation]
    \label{lem:markov-inequality}
    Let $\chi:\R \to \R$ be Lipschitz and nondecreasing, with $\chi(0)=0$. Then, for every real-valued $f\in\H^\beta_v$,
    \[
      a^{k,V}_{t,x}(f,\chi(f))\ge0
    \]
    for almost every $(t,x)$.
  \end{lem}

  \begin{proof}
    By Lemma \ref{lem:L1-lipschitz-calculus}, we have $\chi(f)\in\H^\beta_v$. Since $\chi$ is nondecreasing and $\chi(0)=0$,
    \[
      (r-q)(\chi(r)-\chi(q))\ge0,\qquad r\,\chi(r)\ge0,
      \qquad r,q \in \R.
    \]
    Using $k\ge0$ and $V\ge0$, we deduce
    \begin{align*}
      a^{k,V}_{t,x}(f,\chi(f))
       & =\int_{\R^{2d}}(f-f')(\chi(f)-\chi(f'))k\dd(v,v')
      +\int_{\R^d} Vf \, \chi(f)\dd v
      \ge0.
    \end{align*}
  \end{proof}

\begin{lem}[Weak kinetic chain rule] \label{lem:kinetic-chain-rule}
  Let $f$ be a real-valued function with
  \[
    f\in\C([s,T];\L^2_{x,v}) \cap \L^2((s,T) \times \R^{d}; \H_v^\beta),
    \qquad
    (\partial_t+v\cdot\nabla_x)f
    \in\L^2((s,T)\times \R^{d};\H^{-\beta}_v).  \]
  Let $\chi:\R \to \R$ be Lipschitz with constant $\operatorname{Lip}(\chi)$ and $\chi(0)=0$. We define
  \[
    \Phi(r):=\int_0^r \chi(q)\dd q .
  \]
  Then, for $s\le t_0<t_1\le T$,
  \begin{align*}
     & \int_{\R^{2d}}\Phi(f(t_1)) \dd (x,v)
    -\int_{\R^{2d}}\Phi(f(t_0)) \dd (x,v)
    =\int_{t_0}^{t_1} \langle [(\partial_t+v\cdot\nabla_x)f](\tau), \chi(f(\tau)) \rangle \dd\tau,
  \end{align*}
  where the brackets denote the duality between $\L^2_x\H^{-\beta}_v$ and $\L^2_x\H_v^\beta$.
\end{lem}

\begin{proof}
  We first work on a compact subinterval of $(s,T)$. Standard convolution in the variables $(t,x)$ commutes with $\partial_t + v \cdot \nabla_x$.
  Let $f_n$ denote this regularisation, obtained by convolution with a Friedrichs mollifier, then
  \[
    f_n\to f
    \quad
    \hbox{in }\L^2_{t,x}\H^\beta_v,
    \qquad
    (\partial_t+v\cdot\nabla_x)f_n \to (\partial_t+v\cdot\nabla_x)f
    \quad
    \hbox{in }\L^2_{t,x}\H^{-\beta}_v
  \]
  on this subinterval.
  Moreover, $\chi(f_n)\to \chi(f)$ in $\L^2_{\tau,x}\H^\beta_v$ by Lemma \ref{lem:L1-lipschitz-calculus}.

  For $f_n$ the ordinary chain rule holds for almost every $v$.
  By the definition of $\Phi$, the condition $\chi(0) =  0$, and
  the Lipschitz continuity of $\chi$, we have
  \[
    |\chi(r)| =|\chi(r)-\chi(0)|  \le \operatorname{Lip}(\chi)|r|
  \]
  and
  \[
    |\Phi(r)|  = \left|\int_0^r \chi(q) \dd q\right|
    \le \frac12\operatorname{Lip}(\chi)|r|^2 .
  \]
  Hence, we have $\Phi(f_n(\tau,\cdot,v)) \in \W^{1,1}_x$ for almost every $(\tau,v)$ and
  \[
    \int_{\R^d}v\cdot\nabla_x\Phi(f_n(\tau,x,v)) \dd x=0.
  \]

  Let $0 \le \eta \in \C_c^\infty(\R^d)$ be a cutoff function with $\eta =1$ in $B_1(0)$. We set $\eta_R(v) = \eta(v/R)$, $R\ge1$.
  For fixed large enough $n$, the functions $ \Phi(f_{n})$ are $\C^1$ on $[t_{0},t_{1}]$ valued in $\L^1_{x,v}$ and $\partial_t f_n$ and $\nabla_x f_n$ belong to $\L^2((t_0,t_1) \times \R^{2d})$. As $v$ is bounded on the support of $\eta_R$ we have
  \[
    (\partial_t f_n)\chi(f_n)\eta_R, \; (v \cdot \nabla_xf_n) \chi(f_n) \eta_R \in \L^{1}((t_0,t_1) \times \R^{2d}).
  \]
  Differentiating yields
  \begin{align*}
     & \int_{\R^{2d}} \eta_R\Phi(f_n(t_{1})) \dd (x,v)-\int_{\R^{2d}}\eta_R\Phi(f_n(t_{0}))\dd (x,v)                                                  \\
     & = \int_{t_0}^{t_1} \int_{\R^{2d}} (\partial_{\tau} f_{n})\chi(f_n) \eta_R \dd (x,v) \dd \tau                                                      \\
     & = \int_{t_0}^{t_1} \int_{\R^{2d}} ((\partial_{\tau} + v \cdot \nabla_x) f_{n})\chi(f_n) \eta_R \dd (x,v) \dd \tau                                 \\
     & =\int_{t_0}^{t_1}\big\langle (\partial_\tau+v\cdot\nabla_x) f_n,\eta_R \chi(f_n) \big\rangle_{\L^2_x\H_v^{-\beta} \times \L^2_x\H_v^{\beta}} \dd\tau.
  \end{align*}
  In the last equality we identify the integral, which is over a bounded subset in velocity, with the dual pairing. This can be seen easily for smooth test functions and then by density in the general case.

  Multiplication by $\eta_R$ is uniformly bounded on $\H^\beta_v$ and converges strongly to the identity, so $\eta_R \chi(f_n)\to \chi(f_n) $ in $\L^2_{\tau,x}\H^\beta_v$ as $R \to \infty$.
  The endpoint terms converge by dominated convergence, since
  \[
    |\Phi(r)-\Phi(q)|
    \le\operatorname{Lip}(\chi)(|r|+|q|)|r-q|.
  \]
  Letting $R\to\infty$, we obtain
  \begin{align*}
     & \int_{\R^{2d}} \Phi(f_n(t_1)) \dd (x,v)-\int_{\R^{2d}} \Phi(f_n(t_0)) \dd (x,v)                                                          \\
     & =\int_{t_0}^{t_1}\big\langle(\partial_\tau+v\cdot\nabla_x)f_n,\chi(f_n)\big\rangle_{\L^2_x\H_v^{-\beta} \times \L^2_x\H_v^{\beta}}\dd\tau .
  \end{align*}
  Passing to the limit $n\to\infty$ in the preceding identity yields the claim for compact subintervals $[t_{0},t_{1}]$ of $(s,T)$. Continuity in $\L^2_{x,v}$ and integrability of the bracket allow us to take the limits and we obtain the claim on $[s,T]$.
\end{proof}

  \begin{prop}[Positivity of fundamental solutions]
    \label{prop:L1-positivity}
    Let $0\le s\le t\le T$ and $\psi\in\L^2_{x,v}$ with $\psi\ge0$. Then $\Gamma_{k,V}(t,s)\psi\ge0$.
  \end{prop}

  \begin{proof}
    We set $f(r)=\Gamma_{k,V}(r,s)\psi$. By Lemma \ref{lem:basic-properties}, the function $f$ is real-valued. For $q\in\R$, we write $q^-=\min\{q,0\}$ and take $\chi(q)=q^-$. Then $\chi$ is Lipschitz and nondecreasing, with $\chi(0)=0$, and
    \[
      \Phi(q)=\int_0^q \chi(a) \dd a=\frac12|q^-|^2.
    \]
    Applying Lemmas \ref{lem:kinetic-chain-rule} and \ref{lem:markov-inequality}, we obtain
    \begin{align}
      \frac12\|f^-(t)\|_{\L^2_{x,v}}^2
      -\frac12\|\psi^-\|_{\L^2_{x,v}}^2
       & =-\int_s^t\int_{\R^d}a^{k,V}_{r,x}(f,f^-)\dd x\dd r
      \le0.
    \end{align}
    Since $\psi^-=0$, we conclude that $f^-(t)=0$. This proves the claim.
  \end{proof}

  \begin{prop}[$\L^1$ contractivity of fundamental solutions]
    \label{prop:L1-contraction-L1L2}
    Let $0\le s\le t\le T$ and let $\psi\in\L^1_{x,v}\cap\L^2_{x,v}$ be real-valued. Then
    \begin{equation}
      \label{eq:L1-contraction-L1L2}
      \|\Gamma_{k,V}(t,s)\psi\|_{\L^1_{x,v}}
      \le\|\psi\|_{\L^1_{x,v}}.
    \end{equation}
  \end{prop}

  \begin{proof}
    We set $f(r)=\Gamma_{k,V}(r,s)\psi$. For $\varepsilon>0$, we define
    \[
      \Phi_\varepsilon(q) =  \sqrt{q^2+\varepsilon^2} - \varepsilon,
      \qquad
      \chi_\varepsilon(q) = \Phi_\varepsilon'(q)
      = \frac{q}{\sqrt{q^2+\varepsilon^2}}.
    \]
    The function $\chi_\varepsilon$ is Lipschitz and nondecreasing, with $\chi_\varepsilon(0)=0$. Applying Lemmas~\ref{lem:kinetic-chain-rule} and \ref{lem:markov-inequality}, we deduce
    \[
      \int_{\R^{2d}}\Phi_\varepsilon(f(t))\dd(x,v)
      \le\int_{\R^{2d}}\Phi_\varepsilon(\psi)\dd(x,v)
      \le\|\psi\|_{\L^1_{x,v}}.
    \]
    Here, we used $0\le\Phi_\varepsilon(q)\le|q|$. Since $\Phi_\varepsilon(q)\uparrow|q|$ as $\varepsilon\downarrow0$, monotone convergence gives \eqref{eq:L1-contraction-L1L2}.
  \end{proof}

  \begin{thm}[Extension to $\L^p$ spaces]
    \label{thm:L1-extension}
    Let $1\le p\le\infty$ and $0\le s\le t\le T$. Then $\Gamma_{k,V}(t,s)$ has an extension to $\L^p_{x,v}$. The extension is a positive contraction and satisfies
    \[
      \|\Gamma_{k,V}(t,s)\psi\|_{\L^p_{x,v}}
      \le\|\psi\|_{\L^p_{x,v}},
      \qquad \psi\in\L^p_{x,v}.
    \]
    The same assertions hold for $\widetilde\Gamma_{k,V}(s,t)$.
  \end{thm}

  \begin{proof}
    We first consider real-valued functions. By Proposition~\ref{prop:L1-contraction-L1L2} and density of $\L^1_{x,v}\cap\L^2_{x,v}$ in $\L^1_{x,v}$, the operator extends uniquely to a contraction on $\L^1_{x,v}$. Approximating nonnegative data by nonnegative functions in $\L^1_{x,v}\cap\L^2_{x,v}$ and using Proposition~\ref{prop:L1-positivity}, we deduce positivity of this extension. Consequently,
    \[
      |\Gamma_{k,V}(t,s)\psi|
      \le\Gamma_{k,V}(t,s)|\psi|
    \]
    for complex-valued $\psi\in\L^1_{x,v}$. Thus the complex extension is also contractive. Interpolation with the $\L^2$ contraction from Lemma \ref{lem:basic-properties} gives the  extensions for $1<p<2$.

    The same arguments apply to the backward fundamental solution operators. Using the adjoint relation $\widetilde\Gamma_{k,V}(s,t)=\Gamma_{k,V}(t,s)^*$ from Theorem \ref{thm:representation0T}, we obtain the extensions for $2<p\le\infty$ by duality. Positivity is preserved by these extensions. This completes the proof.
  \end{proof}

From now on, the fundamental solution operators and their extensions act on complex-valued functions.

\subsection{On-diagonal bound and kernel representation}

For the remainder of this section, we take $V=0$.

We set
\[
  Q_\beta=2(1+\beta)d, \qquad q_\beta=\frac{Q_\beta}{2\beta}=\frac{(1+\beta)d}{\beta}.
\]
Here, $Q_\beta$ is the homogeneous dimension of the spatial variables $(x,v)$ for the kinetic scaling $v\sim R^{-1}$, $x\sim R^{-(2\beta+1)}$.

\begin{lem}[Anisotropic Nash inequality]
  For every
  \[
    g\in \L^1(\R^{2d})\cap\L^2(\R^{2d}),
    \qquad D_{\vphantom{t}v}^\beta\, g,\ D_{\vphantom{t}x}^{\gamma_\beta}g\in\L^2(\R^{2d}),
  \]
  we have
  \begin{equation}\label{eq:anisotropic-nash}
    \|g\|_{\L^2_{x,v}}^{2+4\beta/Q_\beta}
    \le C(d,\beta)\|g\|_{\L^1_{x,v}}^{4\beta/Q_\beta}
    \left(\|D_{\vphantom{t}v}^\beta\, g\|_{\L^2_{x,v}}^{\vphantom{4\beta/Q_\beta}2}+\|D_{\vphantom{t}x}^{\gamma_\beta}g\|_{\L^2_{x,v}}^{\vphantom{4\beta/Q_\beta}2}\right).
  \end{equation}
\end{lem}

\begin{proof}
  The proof is the Nash--Stein argument, adapted to the anisotropic kinetic scaling; see \cite{MR100158}.
  Let $R>0$ and consider the set in Fourier variables  
  \[
    E_R=\{(\varphi,\xi): |\xi|\le R,\ |\varphi|\le R^{2\beta+1}\}.
  \]
  Then $|E_R|\le C R^{Q_\beta}$. Since $\|\widehat g\|_{\L^\infty_{x,v}}\le C\|g\|_{\L^1_{x,v}}$,
  \[
    \int_{E_R}|\widehat g|^2\dd (\varphi,\xi)\le C\|g\|_{\L^1_{x,v}}^2R^{Q_\beta}.
  \]
  On $E_R^c$, either $|\xi|>R$ or $|\varphi|>R^{2\beta+1}$. In the second case $|\varphi|^{2\gamma_\beta}>R^{2\beta}$. Hence
  \[
    \int_{E_R^c}|\widehat g|^2\dd (\varphi,\xi)
    \le R^{-2\beta}\int_{\R^{2d}}
    \left(|\xi|^{2\beta}+|\varphi|^{2\gamma_\beta}\right)|\widehat g(\varphi,\xi)|^2 \dd (\varphi,\xi).
  \]
  By Plancherel,
  \[
    \|g\|_{\L^2_{x,v}}^{\vphantom{4\beta/Q_\beta}2}
    \le C\|g\|_{\L^1_{x,v}}^{\vphantom{4\beta/Q_\beta}2} \, R^{Q_\beta}
    +R^{-2\beta}\left(\|D_{\vphantom{t}v}^\beta\, g\|_{\L^2_{x,v}}^{\vphantom{4\beta/Q_\beta}2}+\|D_{\vphantom{t}x}^{\gamma_\beta}g\|_{\L^2_{x,v}}^{\vphantom{4\beta/Q_\beta}2} \right).\]
  Optimising in $R$ gives \eqref{eq:anisotropic-nash}.
\end{proof}

\begin{prop}[Hypercontractivity of fundamental solutions]
  \label{prop:nash-L1-L2}
  There is a constant $C<\infty$, depending only on $d,\beta,\lambda_{k},\Lambda_{k}$, such that, for $0 < t-s \le 1$,
  \begin{equation}
    \label{eq:L1-L2-nash}
    \|\Gamma_{k,0}(t,s)\|_{\L^1_{x,v}\to\L^2_{x,v}}
    +\|\widetilde\Gamma_{k,0}(s,t)\|_{\L^1_{x,v}\to\L^2_{x,v}}
    \le C(t-s)^{-q_\beta/2}.
  \end{equation}
  Consequently,
  \begin{equation}
    \label{eq:L2-Linfty-nash}
    \|\Gamma_{k,0}(t,s)\|_{\L^2_{x,v}\to\L^\infty_{x,v}}
    +\|\widetilde\Gamma_{k,0}(s,t)\|_{\L^2_{x,v}\to\L^\infty_{x,v}}
    \le C(t-s)^{-q_\beta/2}.
  \end{equation}
\end{prop}

\begin{proof}
 By \eqref{eq:local-kernel-coercivity} with $V=0$, the forward and backward transfer-of-regularity estimates
    in Proposition~\ref{prop:transferofregularity} on intervals
    of length at most one have constants depending only on
    $d,\beta,\lambda_k,\Lambda_k$.

  Let $\psi\in\L^1_{x,v}\cap\L^2_{x,v}$, put $f(\tau)=\Gamma_{k,0}(\tau,s)\psi$,
  \[
    Y(\tau) = \|f(\tau)\|_{\L^2_{x,v}}^2,
    \qquad M=\|\psi \|_{\L^1_{x,v}},
    \qquad \delta = \frac{2\beta}{Q_\beta}=\frac1{q_\beta}.
  \]
  Theorem \ref{thm:L1-extension} with $V=0$ gives $\|f(\tau)\|_{\L^1_{x,v}}\le M$. Moreover, \eqref{eq:L2-energy} with zero source and the nonnegativity of $a^{k,0}$ show that $Y$ is nonincreasing. Let $ s\le t_{0}<t_{1}\le t$.
  By \eqref{eq:anisotropic-nash}, 
  for almost every
  $\tau \in (t_{0},t_{1})$,
  \[
    Y(\tau)^{1+\delta} \le C M^{2\delta}
    \left(
    \|D_{\vphantom{t}v}^\beta\, f(\tau)\|_{ \L^2_{x,v}}^2
    +\|D_{\vphantom{t}x}^{\gamma_\beta}f(\tau)\|_{\L^2_{x,v}}^2
    \right).
  \]
  Integrating this inequality and using \eqref{eq:forward-transfert} yields
  \begin{equation*}
    \int_{t_{0}}^{t_{1}}Y(\tau)^{1+\delta} \dd\tau
    \le C M^{2\delta}Y(t_{0}).
  \end{equation*}
  Since $Y$ is nonincreasing,
  \begin{equation}
    \label{eq:nash-one-step}
    (t_{1}-t_{0})Y(t_{1})^{1+\delta}
    \le C M^{2\delta}Y(t_{0}).
  \end{equation}
  Let $h=t-s$, $\rho=(1+\delta)^{-1}$, and $\tau_j = s+2^{-j}h$, so that $\tau_0 = t$ and $\tau_j \downarrow s$. Applying \eqref{eq:nash-one-step} on $[\tau_{j+1},\tau_j]$, with $\Delta_j=\tau_j-\tau_{j+1}$, gives
  \[
    Y(\tau_j)
    \le \left(CM^{2\delta}\Delta_j^{-1}\right)^\rho
    Y(\tau_{j+1})^\rho .
  \]
  Iteration from $j=0$ to $N-1$ gives
  \[
    Y(t)
    \le Y(\tau_N)^{\rho^N}
    \prod_{j=0}^{N-1}
    \left(CM^{2\delta}\Delta_j^{-1}\right)^{\rho^{j+1}}.
  \]
  Because $Y(\tau_N)\le Y(s) < \infty$, the first factor has $\limsup$ of at most $1$ as $N\to\infty$. Moreover,
  \[
    \sum_{j=0}^\infty\rho^{j+1}=\frac1\delta,
    \qquad
    \sum_{j=0}^\infty (j+1)\rho^{j+1}<\infty .
  \]
  Since $\Delta_j^{-1}=2^{j+1}h^{-1}$, 
  the product converges as $N\to\infty$ to 
  $C^{1/\delta} 2^{(1+\delta)/\delta} M^2h^{-1/\delta}$. Passing to the limit therefore gives
  \[
    Y(t)\le C M^2 h^{-1/\delta}
    = C\|\psi\|_{\L^1_{x,v}}^2(t-s)^{-q_\beta}.
  \]
  Taking square roots proves \eqref{eq:L1-L2-nash} for $\psi \in \L^1_{x,v} \cap \L^2_{x,v}$. Finally, density of $\L^1_{x,v} \cap \L^2_{x,v}$ in $\L^1_{x,v}$ and the $\L^1_{x,v}$-extension from Theorem~\ref{thm:L1-extension} give the $\L^1_{x,v} \to \L^2_{x,v}$ bound.

  The same proof applies in the backward case since we have observed its $L^1_{x,v}$ contractivity in Theorem~\ref{thm:L1-extension}.

  This proves the estimate for
  $\widetilde\Gamma_{k,0}(s,t)$ in \eqref{eq:L1-L2-nash}. The
  estimates \eqref{eq:L2-Linfty-nash} then follow from the
  adjoint relation
  $\widetilde\Gamma_{k,0}(s,t)=\Gamma_{k,0}(t,s)^*$ and duality.
\end{proof}

\begin{thm}[Kernel representation with on-diagonal upper bound]
  \label{thm:nash-kernel}
  For every $0 < t-s\le 1$, the $\L^1$-extension of $\Gamma_{k,0}(t,s)$ is bounded from $\L^1_{x,v}$ to $\L^\infty_{x,v}$, and
  \begin{equation}   \label{eq:L1-Linfty-nash}
    \|\Gamma_{k,0}(t,s)\|_{\L^1_{x,v}\to\L^\infty_{x,v}}
    \le C_{1}(t-s)^{-q_\beta},
  \end{equation}
  where $C_{1}<\infty$ depends only on $d,\beta,\lambda_{k},\Lambda_{k}$. Consequently $\Gamma_{k,0}(t,s)$ is induced by a measurable kernel, denoted by $\Gamma_{k,0}(t,x,v,s,y,w)$,
  in the sense that
  \begin{equation}
    \label{eq:kernel-representation-L1}
    (\Gamma_{k,0}(t,s)\psi)(x,v)
    =\int_{\R^{2d}}\Gamma_{k,0}(t,x,v,s,y,w)\psi(y,w) \dd (y,w)
  \end{equation}
  for every $\psi\in\L^1_{x,v}$ and for almost every $(x,v)$. The kernel may be chosen nonnegative and satisfies the global bound
  \begin{equation}
    \label{eq:on-diagonal-bound}
    \esssup_{x,v,y,w\in\R^d}
    \Gamma_{k,0}(t,x,v,s,y,w)
    \le C_{1}(t-s)^{-q_\beta}
  \end{equation}
  for $0<t-s\le1$.

  The same kernel represents the $\L^2$-fundamental solution operator. The kernels satisfy,
  \begin{equation}
    \label{eq:kernel-adjoint}
    \Gamma_{k,0}(t,x,v,s,y,w)
    =\widetilde\Gamma_{k,0}(s,y,w,t,x,v),
  \end{equation}
  and, if $s<r<t$,
  \begin{equation}
    \label{eq:kernel-CK}
    \Gamma_{k,0}(t,x,v,s,y,w)
    =\int_{\R^{2d}}\Gamma_{k,0}(t,x,v,r,z,\xi)\Gamma_{k,0}(r,z,\xi,s,y,w) \dd (z,\xi) .
  \end{equation}
\end{thm}

\begin{proof}
  Let $m=(s+t)/2$. By the Chapman--Kolmogorov relation in
    Theorem~\ref{thm:representation0T} and Proposition~\ref{prop:nash-L1-L2},
  \begin{align*}
    \|\Gamma_{k,0}(t,s)\|_{\L^1_{x,v}\to\L^\infty_{x,v}}
     & \le \|\Gamma_{k,0}(t,m)\|_{\L^2_{x,v}\to\L^\infty_{x,v}}
    \|\Gamma_{k,0}(m,s)\|_{\L^1_{x,v}\to\L^2_{x,v}}             \\
     & \le C^2(t-m)^{-q_\beta/2}(m-s)^{-q_\beta/2}
    = 2^{q_\beta}C^2(t-s)^{-q_\beta}.
  \end{align*}
  This proves \eqref{eq:L1-Linfty-nash}.

  The Dunford--Pettis representation theorem, applied to the bounded operator $\Gamma_{k,0}(t,s):\L^1_{x,v}\to\L^\infty_{x,v}$, gives a measurable kernel $\Gamma_{k,0}(t,x,v,s,y,w)\in\L^\infty(\R^{2d}_{x,v}\times\R^{2d}_{y,w})$ such that \eqref{eq:kernel-representation-L1} holds for $\psi\in\L^1_{x,v}$, and the $\L^\infty_{x,v,y,w}$-norm of the kernel is the $\L^1_{y,w}\to\L^\infty_{x,v}$ operator norm. Since $\Gamma_{k,0}(t,s)$ is a positive operator on $\L^1_{x,v}$, the kernel may be chosen nonnegative. This proves \eqref{eq:on-diagonal-bound}.

  It remains to check that the same kernel represents the $\L^2$-operator. This is classical but we include an argument for the comfort of the reader. The $\L^2_{x,v}\to\L^\infty_{x,v}$ bound in \eqref{eq:L2-Linfty-nash} implies that, after removing one null set in $(x,v)$, the functional
  \[
    \psi\mapsto (\Gamma_{k,0}(t,s)\psi)(x,v),
    \qquad \psi\in\L^1_{x,v}\cap\L^2_{x,v},
  \]
  is bounded on $\L^2_{x,v}$, with norm at most $C(t-s)^{-q_\beta/2}$. By the Riesz representation theorem and the $\L^1$-kernel representation, for almost every $(x,v)$,
  \begin{equation}
    \label{eq:kernel-L2}
    \left(\int_{\R^{2d}}|\Gamma_{k,0}(t,x,v,s,y,w)|^2 \dd (y,w)\right)^{1/2}
    \le C(t-s)^{-q_\beta/2}.
  \end{equation}
  Thus \eqref{eq:kernel-representation-L1} is absolutely convergent by Cauchy Schwarz inequality for every $\psi\in\L^2_{x,v}$, for almost every $(x,v)$, and gives the $\L^2$ action of $\Gamma_{k,0}(t,s)$.

  Applying the same construction to $\widetilde\Gamma_{k,0}(s,t)$ gives a kernel for the backward operator and the analogous estimate. The adjoint identity \eqref{eq:kernel-adjoint} follows by testing the operator identity $\widetilde{\Gamma}_{k,0}(s,t)=\Gamma_{k,0}(t,s)^*$ against functions in $\L^1_{x,v}\cap\L^2_{x,v}$ and using uniqueness of the integral kernels. Combining this identity with the estimate for the backward kernel gives the estimate
  \begin{equation}
    \label{eq:kernel-L2-adj}
    \left(\int_{\R^{2d}}|\Gamma_{k,0}(t,x,v,s,y,w)|^2 \dd (x,v)\right)^{1/2}
    \le C(t-s)^{-q_\beta/2}
  \end{equation}
  for almost every $(y,w)$.

  The operator Chapman--Kolmogorov identity gives \eqref{eq:kernel-CK}. The integral is absolutely convergent for almost every $(x,v,y,w)$ because \eqref{eq:kernel-L2} and \eqref{eq:kernel-L2-adj} imply that both factors are in $\L^2_{z,\xi}$.
\end{proof}

\section{Perturbed $\L^p$ bounds by Davies method}

We implement Davies perturbation trick. In the kinetic setting the exponentials need to be adapted to the kinetic transport term as observed in \cite{AIN2}.

Throughout this section, $k$ is measurable, symmetric and nonnegative, and satisfies \eqref{eq:upperbound} and \eqref{eq:local-stable-lower}.
Let $h\in \W^{1,\infty}(\R^{2d};\R)$, the space of real-valued bounded Lipschitz functions. We set
\begin{align}
   & h_t(x,v)=h(x-tv,v),
  \\
   & \Gamma^h_{k,0}(t,s)=e^{h_t}\Gamma_{k,0}(t,s)e^{-h_s},
  \\
   & M(\theta)=e^\theta+e^{-\theta}-2,
  \\
   & \begin{aligned}
       \omega_{k,h}(t)
       ={} \esssup_{(x,v')\in\R^{2d}}\int_{\R^d}
       k(t,x,v,v') \
       M\left(h_t(x,v)-h_t(x,v')\right)\dd v.
     \end{aligned}
\end{align}
We call $\Gamma^h_{k,0}(t,s)$ the perturbed fundamental solution operator, that is the fundamental solution operator perturbed by multiplication by exponential functions that are constant along transport, as $(\partial_t+v\cdot\nabla_x)e^{h_{t}(x,v)}=0$. We set $\Gamma^h=\Gamma^h_{k,0}$ for the fixed kernel of \eqref{eq:stable-comparable}. 

As is common in the study of diffusion equations, $\L^p$-bounds for  the perturbed operator depend on precise non-linear inequalities on real numbers. The non-local diffusion forces a control of the Lipschitz function $h$ with the  factor $M$  instead of taking partial derivatives.  This is similar to calculations made in discrete settings and we got inspired in part by \cite[Proof of Lemma 2.2]{MR5068936} for the next inequality.

\begin{lem} [An auxiliary inequality]\label{lem:auxiliary-lower-bound}
  For $p\ge2$, $u_1,u_2\ge0$, and $\theta\in\R$, we have
  \begin{equation}
    \label{eq:auxiliary-lower-bound}
    u_1^p+u_2^p-e^\theta u_1u_2^{p-1}-e^{-\theta}u_2u_1^{p-1}
    \ge 
    -\left(\frac{p+2}{4}\right) M(\theta)(u_1^p+u_2^p).
  \end{equation}
\end{lem}

\begin{proof}
  We write
  \begin{align*}
    R(\theta)  =e^\theta-1-\theta \ge0, &
                                        & R(\theta) & +R(-\theta)                =M(\theta), \\
    A      =(u_1-u_2)(u_1^{p-1}-u_2^{p-1}),     &
                                        & B         & =u_1u_2^{p-1}-u_2u_1^{p-1}.
  \end{align*}
  The left-hand side of \eqref{eq:auxiliary-lower-bound} is therefore
  \begin{equation*}
    \cL =A-\theta B-R(\theta)u_1u_2^{p-1}-R(-\theta)u_2u_1^{p-1}.
  \end{equation*}

  We first claim that
  \begin{equation*}
    B^2\le(p-2)(u_1^p+u_2^p)A.
  \end{equation*}
  The cases $u_1=u_2$, $u_1u_2=0$ and $p=2$ are immediate. The inequality is symmetric in $u_1$ and $u_2$. Thus, it suffices to consider $0<u_2<u_1$. We set $r=u_2/u_1 \in (0,1)$ and deduce
  \begin{align*}
    \frac{B^2}{(u_1^p+u_2^p)A}
     & =\frac{r^2(1-r^{p-2})^2}
        {(1+r^p)(1-r)(1-r^{p-1})}   \\
     & =\frac{r-r^{p-1}}{1-r^{p-1}}
    \frac{r-r^{p-1}}{(1+r^p)(1-r)}
    \le\frac{p-2}{1+r^p}\le p-2.
  \end{align*}
  Here, we used convexity in the form of $r^{p-1} \ge 1+(p-1)(r-1)$.

  The power series representation gives
  \begin{equation}
    \theta^2\le M(\theta).
  \end{equation}

  Using Young's inequality gives
  \begin{equation}
    R(\theta)u_1u_2^{p-1}+R(-\theta)u_2u_1^{p-1}
    \le {M(\theta)}(u_1^p+u_2^p).
  \end{equation}

  We set $D=M(\theta)(u_1^p+u_2^p)$, then
  \begin{align*}
    \cL
     & \ge A-\sqrt{(p-2)DA}-D                                                          \\
     & =\left(\sqrt{A}-\frac{\sqrt{(p-2)D}}{2}\right)^2-\left(\frac{p-2}{4}+1\right)D.
  \end{align*}
\end{proof}

\begin{prop}[Perturbed $\L^p$ bound]
  \label{prop:perturbed-Lp}
  For every function $h \in \W^{1,\infty}(\R^{2d};\R)$, $2\le p < \infty$ and $0\le s \le t\le T$ we have
  \[
    \|\Gamma^h_{k,0}(t,s)\|_{\L^p_{x,v}\to\L^p_{x,v}}
    \le\exp\!\left( \frac{p+2}{2}  \int_s^t\omega_{k,h}(r)\dd r\right).
  \]
\end{prop}

\begin{proof}

  Let us first consider $\psi \in \L^2_{x,v} \cap \L^\infty_{x,v}$, $\psi\ge 0$. We set
  \begin{equation*}
    f(r,x,v) =e^{h_r(x,v)}
    \left[\Gamma_{k,0}(r,s)(e^{-h_s}\psi)\right](x,v), \qquad Y_p(r) =\int_{\R^{2d}}|f(r,x,v)|^p\dd (x,v).
  \end{equation*}
  Observe that  $e^{-h_r}f$ is a nonnegative weak solution to $(\partial_r+v\cdot\nabla_x)(e^{-h_r}f)+\cA_k(e^{-h_r}f)=0$ when $r>s$. By construction, $e^{-h_r}f$ belongs to the regularity class stated in Lemma~\ref{lem:kinetic-chain-rule}. So does $f$ as $e^{h_r}$ is bounded and Lipschitz on $[s,T]\times \R^{2d}$ with $(\partial_r+v\cdot\nabla_x)e^{h_r}=0$. 
  We have
  \begin{align*}
    (\partial_r+v\cdot\nabla_x)f
     & =-e^{h_r}\cA_k(e^{-h_r}f).
  \end{align*}
  By Theorem~\ref{thm:L1-extension}, since $e^{-h_s}\psi$ is bounded on $\R^{2d}$, we also know that $e^{-h_r}f$ is bounded on $[s,T]\times\R^{2d}$, and so is $f$.
  Applying Lemma~\ref{lem:kinetic-chain-rule} to $f$ with a Lipschitz function $\chi$ that agrees with $\zeta \mapsto p\zeta^{p-1}$ on $[0,\|f\|_\infty]$ and vanishes at zero, we obtain that 
  $Y_p$ is absolutely continuous and 
  \begin{align*}
    \frac{\dd}{\dd r}Y_p(r)
     & =-p\int_{\R^d}a^{k,0}_{r,x}
    (e^{-h_r}f,e^{h_r}f^{p-1})\dd x
  \end{align*}
  for almost every $r \in (s,T)$.

  Let us now fix $(r,x)$ and abbreviate
  \[
    u_1=f(r,x,v),\qquad u_2=f(r,x,v'),\qquad
    \theta=h_r(x,v')-h_r(x,v).
  \]
  Lemma \ref{lem:auxiliary-lower-bound} yields
  \begin{align*}
     & \left(e^{-h_r(x,v)}u_1-e^{-h_r(x,v')}u_2\right)
    \left(e^{h_r(x,v)}u_1^{p-1}-e^{h_r(x,v')}u_2^{p-1}\right)         \\
     & \qquad=u_1^p+u_2^p-e^\theta u_1u_2^{p-1}-e^{-\theta}u_2u_1^{p-1}      \\
     & \qquad\ge-\left(\frac{p+2}{4} \right) M(\theta)(u_1^p+u_2^p).
  \end{align*}
  Since $k \ge 0$, integration in $(v,v')$ gives
  \begin{align*}
     & a^{k,0}_{r,x}(e^{-h_r}f,e^{h_r}f^{p-1})              \\
     & \ge-\left(\frac{p+2}{4}\right)\int_{\R^d\times\R^d}
    M\left(h_r(x,v')-h_r(x,v)\right)                         \\
     & \hspace{8em}\times\left(f(r,x,v)^p+f(r,x,v')^p\right)
    k(r,x,v,v')\dd (v,v').
  \end{align*}
  By the evenness of $M$ and the symmetry of $k$, one bounds the integral by
  \begin{equation*}
    2 \ \omega_{k,h}(r)\int_{\R^d}f(r,x,v)^p\dd v.
  \end{equation*}
  Hence,
  \begin{equation*}
    \frac{\dd}{\dd r}Y_p(r) \le p \left( \frac{p+2}{2} \right)\omega_{k,h}(r)Y_p(r).
  \end{equation*}
  Applying Gronwall's inequality and taking the $p$-th root gives the stated bound.
  Positivity shows $|\Gamma^h_{k,0}(t,s) \psi| \le \Gamma^h_{k,0}(t,s) |\psi|$ for complex-valued $\psi$, and thus we obtain the bound also for the complex extension. Finally, we use density of $\L^2_{x,v}\cap \L^\infty_{x,v}$ in $\L^p_{x,v}$ to deduce the claim.
\end{proof}

\section{Interpolation and iteration of the perturbed $\L^p$ bound}

Next, we combine the perturbed $\L^p$-estimate with the unweighted bounds for the fundamental solution operators with an iteration reminiscent of the one by Moser.
We recall $q_\beta = \frac{(2+2\beta)d}{2\beta}$. We restrict to $T=1$.

\begin{prop}[Perturbed $\L^1/\L^2-\L^\infty$]
  \label{prop:perturbed-L2-Linfty}
  Let $0\le s<t\le1$ and $h\in\W^{1,\infty}(\R^{2d};\R)$. Then
  \begin{equation}
    \label{eq:iteration-L2-infinity}
    \|\Gamma^h_{k,0}(t,s)\|_{\L^2_{x,v}\to\L^\infty_{x,v}}
    \le C(t-s)^{-\frac{q_\beta}{2}}
    \exp\!\left( \frac 3 2\Omega_{k,h}(s,t)(t-s)\right)
  \end{equation}
  and
  \begin{equation}
    \label{eq:iteration-L1-infty}
    \|\Gamma^h_{k,0}(t,s)\|_{\L^1_{x,v}\to\L^\infty_{x,v}}
    \le C(t-s)^{- q_\beta}
    \exp\!\left( \frac 3 2\Omega_{k,h}(s,t)(t-s)\right)
  \end{equation}
  with $\Omega_{k,h}(s,t)$ given by \eqref{eq:iteration-Omega}.
  The constant $C$ depends only on $d,\beta, \lambda_{k}, \Lambda_{k}$.
\end{prop}

\begin{proof}
  Proposition~\ref{prop:perturbed-Lp} gives, using $2\le p$,
  \begin{equation}   \label{eq:perturbed-Lp-iteration}
    \|\Gamma^{2h}_{k,0}(t,s)\|_{\L^p_{x,v}\to\L^p_{x,v}}
    \le\exp\!\left(p\int_s^t\omega_{k,2h}(r)\dd r\right).
  \end{equation}
  The Riesz-Thorin interpolation of the on-diagonal $\L^1_{x,v}\to \L^\infty_{x,v}$ bound in Theorem~\ref{thm:nash-kernel} with the $\L^\infty_{x,v}$ boundedness in Theorem~\ref{thm:L1-extension} yields
  \begin{equation}
    \label{eq:iteration-nash-endpoint}
    \|\Gamma_{k,0}(t,s)\|_{\L^p_{x,v}\to\L^\infty_{x,v}}
    \le C_1^{\frac 1p}(t-s)^{-\frac{q_\beta}{p}}
  \end{equation}
  for some constant $C_1\ge 1$ independent of $p$ and $h$.

  Applying Lemma \ref{lem:Stein-positive-kernel} with \eqref{eq:iteration-nash-endpoint} and \eqref{eq:perturbed-Lp-iteration} gives
  \begin{equation}
    \label{eq:iteration-gain}
    \|\Gamma^h_{k,0}(t,s)\|_{\L^p_{x,v}\to\L^{2p}_{x,v}}
    \le C_1^{\frac{1}{2p}}(t-s)^{-\frac{q_\beta}{2p}}
    \exp\!\left(\frac{p}2\int_s^t\omega_{k,2h}(r)\dd r\right).
  \end{equation}

  We next iterate \eqref{eq:iteration-gain}. The key is to choose the time intervals in a way so that the lengths decrease faster than the exponents increase.
  We set
  \begin{align}
    p_j & =2^{j+1},\qquad j=0,1,\ldots,
    \\
    t_0 & =s,\qquad
    t_{j+1}=t_j+\frac{3(t-s)}{4^{j+1}},
    \qquad\Delta_j=\frac{3(t-s)}{4^{j+1}},
    \\
    t_j & =s+\sum_{i=0}^{j-1}\frac{3(t-s)}{4^{i+1}}
    =s+(t-s)(1-4^{-j})\longrightarrow t.
  \end{align}
  Equation \eqref{eq:iteration-gain} gives
  \begin{equation}
    \|\Gamma_{k,0}^h(t_{j+1},t_j)\|_{\L^{p_j}_{x,v}\to\L^{p_{j+1}}_{x,v}}
    \le C_1^{\frac{1}{2p_j}}\Delta_j^{-\frac{q_\beta}{2 p_j}}
    \exp\!\left(\frac{p_{j}}2\int_{t_j}^{t_{j+1}}\omega_{k,2h}(r)\dd r\right).
  \end{equation}
  The exponential weights cancel at each intermediate time and this yield
  \begin{equation*}
    \Gamma_{k,0}^h(t,s) = \Gamma_{k,0}^h(t,r) \Gamma_{k,0}^h(r,s), \qquad s<r<t.
  \end{equation*}
  For $\psi \in \L^2_{x,v}$ and $n \ge 1$, iteration yields
  \begin{equation*}
    \|\Gamma_{k,0}^h(t_n,s)\psi\|_{\L^{p_n}_{\vphantom{2}x,v}}
    \le C_1^{\sum_{j=0}^{n-1}\frac1{2p_j}}
    \prod_{j=0}^{n-1}\Delta_j^{-\frac{q_\beta}{2p_j}}
    \exp\!\left(\frac 1 2 \Omega_{k,h}(s,t)
    \sum_{j=0}^{n-1}p_j\Delta_j\right)\|\psi\|_{\L^2_{\vphantom{2}x,v}},
  \end{equation*}
  with
  \begin{equation}
    \label{eq:iteration-Omega}
    \Omega_{k,h}(s,t):=\esssup_{r\in(s,t)}\omega_{k,2h}(r).
  \end{equation}
  The corresponding series in the exponents are
  \begin{equation*}
    \sum_{j=0}^{\infty}\frac1{2p_j} = \frac{1}{2}, \qquad \sum_{j=0}^{\infty}\frac{j+1}{2p_j} = 1, \qquad \sum_{j=0}^{\infty}p_j\Delta_j = 3(t-s),
  \end{equation*}
  hence
  \begin{equation}
    \label{eq:iteration-product}
    \|\Gamma_{k,0}^h(t_n,s)\psi\|_{\L^{p_n}_{\vphantom{2}x,v}}
    \le C_1^{\frac1{2}} 4^{q_{\beta}} (3(t-s))^{-\frac{q_{\beta}}2}
    \exp\!\left( \frac 3 2\Omega_{k,h}(s,t)(t-s)\right)
    \|\psi\|_{\L^2_{\vphantom{2}x,v}}.
  \end{equation}
  As $\Gamma_{k,0}^h(t_n,s)\psi$ converges to $\Gamma_{k,0}^h(t,s)\psi$ in $\L^2_{x,v}$, a measure theoretical argument using subsequences shows that $\Gamma_{k,0}^h(t,s)\psi$ is bounded and $\|\Gamma_{k,0}^h(t,s)\psi\|_{\L^{\infty}_{\vphantom{2}x,v}}$ is controlled by the right hand side of \eqref{eq:iteration-product}, which is \eqref{eq:iteration-L2-infinity}.

  The same argument applies to the backward problems with $h$ replaced by $-h$ with the same upper bound. By duality this gives us the corresponding $\L^1_{x,v} \to \L^2_{x,v}$ bound for $\Gamma_{k,0}^h(t,s)$.
  We consider the intermediate time $m = \frac{s+t}{2}$ to deduce
  \begin{align*}
    \|\Gamma_{k,0}^h(t,s)\|_{\L^1_{x,v}\to\L^{\infty}_{x,v}} & \le \|\Gamma_{k,0}^h(t,m)\|_{\L^2_{x,v}\to\L^{\infty}_{x,v}} \|\Gamma_{k,0}^h(m,s)\|_{\L^1_{x,v}\to\L^{2}_{x,v}}
    \\
                                                             & \lesssim(t-s)^{-q_\beta} e^{\frac 3 2 \Omega_{k,h}(s,t) (t-s)}
  \end{align*}
  and thus the claim.
\end{proof}

\section{Optimisation of the Lipschitz weight and rapid off-diagonal decay}

At this point, we can optimise the choice of $h$. We need to follow the dependences in the Lipschitz constants for $h$.

\begin{lem}
  \label{lem:boundomegah} Assume that $h\in \W^{1,\infty}(\R^{2d};\R)$ with for $\delta\ge 0$,
  $\|\nabla_{\!v}\,h\|_{\infty} \le 2\delta$ and $\|\nabla_{\!x}\,h\|_{\infty} \le \delta$. Then
  \begin{equation}
    \esssup_{s\in(0,1)} \, \omega_{k_{\le 1},h}(s) \le C(d,\beta) \Lambda_{k} \, \delta^2\, e^{3\delta}.
  \end{equation}
\end{lem}

\begin{proof}
  Remark that $M(\theta)= |e^\theta-1|^2 \, e^{-\theta}$. Setting \begin{align*}
    m_{}=|k_{\le 1}|\frac{|e^{|h_s'-h_s^{\vphantom{'}}|}-1|^2}
                    {e^{|h_s'-h_s^{\vphantom{'}}|}},
  \end{align*}
  we have
  \begin{equation*}
    \omega_{k_{\le 1},h}(s)=\esssup_{(x,v')\in\R^{2d}}
    \int_{\R^d}m_{}(s,x,v,v') \dd v.
  \end{equation*}
  Note that
  $ |e^\theta-1|^2 \, e^{-\theta} \le \theta^2\, e^\theta$ for $\theta>0$.
  Also for $s\in [0,1]$,
  $$
    |h_{s}(x,v')-h_{s}(x,v)| \le (s\|\nabla_{\!x}\,h\|_{\infty}+ \|\nabla_{\!v}\,h\|_{\infty})|v-v'| \le 3\delta |v-v'|.
  $$
  Thus,
  we have
  \begin{align*}
    \int_{\R^d} m_{}(s,x,v,v') \dd v & \le \int_{|v-v'|\le 1} (3\delta |v-v'|)^2 \, e^{3\delta |v-v'|} | k(s,x,v,v')|\dd v
    \\
                                     & \le \Lambda_{k} (3\delta)^2 \, e^{3\delta} \int_{|v-v'|\le 1}  |v-v'|^{-d-2\beta+2} \dd v.
  \end{align*}
\end{proof}

\begin{rem} It is possible to obtain a bound for $\omega_{k,h}$ itself. But this one involves also a control on the $\L^\infty$ norm of $h$ due to the large-jumps. More precisely, if we quantify $\|h\|_{\infty}\le \delta\rho$, we can obtain
  \begin{equation}
    \esssup_{s\in(0,1)} \, \omega_{k,h}(s) \le C(d,\beta) \Lambda_{k} \, e^{2\delta\rho} \rho^{-2\beta}
  \end{equation}
  when $\delta \rho$ is large enough. This bound seems insufficient for the optimisation below.
\end{rem}

  \begin{prop}[Rapid kernel decay for small-jump fundamental solutions] \label{prop:small-rapid}
    Let $k$ be measurable, symmetric and nonnegative, and satisfy
    \eqref{eq:upperbound}, \eqref{eq:local-stable-lower} and
    \begin{equation} \label{eq:small-jump-support}
      k(t,x,v,v')=0\qquad\text{for }|v-v'|>1.
    \end{equation}
    For every $N>0$, there is a constant $C_N=C(N,d,\beta,\lambda_k,\Lambda_k)<\infty$ such that,
    \begin{equation}
      \Gamma_{k,0}(1,x,v,0,y,w)
      \le C_N(1+|x-y-w|+|v-w|)^{-N}
    \end{equation}
    for almost all $(x,v,y,w)\in\R^{4d}$.
  \end{prop}

\begin{proof}
  For $(x,v,y,w)$, we set $\rho=|x-y-w|+|v-w|$. We have nothing to do if $\rho$ is small by Theorem~\ref{thm:nash-kernel}. We assume $\rho\ge\rho_{0}$ large enough, to be determined.
  In terms of the induced nonnegative kernel the estimate \eqref{eq:iteration-L1-infty} is the inequality
  \begin{equation*}
    \Gamma_{k,0}(1,x,v,0,y,w) \le C \exp \left(h_0(y,w)-h_1(x,v)+\frac 3 2\, \Omega_{k,h}(0,1) \right),  
  \end{equation*}
  valid for almost every $(x,v,y,w) \in \R^{2d}$, for all $h\in \W^{1,\infty}(\R^{2d};\R)$ with $C$ depending on $d,\beta, \lambda_{k},\Lambda_{k}$.
  We fix  $(x,v,y,w)$ for which this holds. We let $\delta>0$ and pick
  \[h(x',v')=\delta\min(|x'+v'-y-w|+|v'-w|, \rho), \qquad (x',v')\in \R^{2d},\]
  so that $\|\nabla_{\!v}\,h\|_{\infty} \le 2\delta$ and $\|\nabla_{\!x}\,h\|_{\infty} \le \delta$, $h_{0}(y,w)=h(y,w)=0$ and $h_1(x,v)=h(x-v,v)=\delta\rho$.
  Since $k_{\le1}=k$ by \eqref{eq:small-jump-support}, Lemma~\ref{lem:boundomegah} and the definition of $\Omega_{k,h}$ gives us
  \begin{equation*}
    \Gamma_{k,0}(1,x,v,0,y,w) \le C \exp \left(-\delta\rho+C'\delta^2e^{6\delta} \right).
  \end{equation*}
  We still have $\delta>0$ as a free parameter. Pick $\gamma\in (0,1) $ and $6\delta=\ln(\rho^\gamma)$ assuming $\rho>1$. Then
  \[-\delta\rho+C'\delta^2e^{6\delta}= -\frac\gamma 6 \rho\ln(\rho) \left(1- \frac{C'\gamma} 6 \rho^{\gamma-1}\ln(\rho)\right)\le -\frac\gamma {12} \rho\ln(\rho)
  \]
  when $\frac{C'\gamma} 6 \rho^{\gamma-1}\ln(\rho)\le \frac 1 2$. This implies the claim.
\end{proof}

\begin{rem}
  Note that we even obtain super-exponential decay when the kernel is supported on small-jumps only.
\end{rem}

  \begin{prop}[Comparison with a nonnegative potential]
    \label{prop:nonnegative-potential}
    Let $T>0$. Let $k$ be measurable, symmetric and nonnegative, and satisfy \eqref{eq:upperbound} and  \eqref{eq:local-stable-lower}. Let $V\in\L^\infty((0,T)\times\R^{2d})$ be nonnegative. Then, for $0\le s\le t\le T$ and $\psi\in\L^2_{x,v}$,
    \begin{equation}
      \label{eq:potential-comparison}
     |\Gamma_{k,V}(t,s)\psi|
      \le \Gamma_{k,V}(t,s)|\psi|
      \le\Gamma_{k,0}(t,s)|\psi|.
    \end{equation}
    For $0<t-s\le1$, the extensions from Theorem \ref{thm:L1-extension} have measurable kernels representing their $\L^1$ and $\L^2$ actions, with
    \begin{equation}
      \label{eq:potential-kernel-comparison}
      0\le\Gamma_{k,V}(t,x,v,s,y,w)
      \le\Gamma_{k,0}(t,x,v,s,y,w)
    \end{equation}
    for almost every $(x,v,y,w)$. These kernels satisfy the adjoint and Chapman--Kolmogorov identities in Theorem \ref{thm:nash-kernel}.

    All $\L^p\to\L^q$ and kernel upper bounds proved above hold with $\Gamma_{k,0}$ replaced by $\Gamma_{k,V}$, with the same constants and on the same time ranges. Under assumption \eqref{eq:small-jump-support} we have the rapid decay in Proposition \ref{prop:small-rapid}.
  \end{prop}

  \begin{proof}
    Let us first consider nonnegative data. Let $\psi\in\L^2_{x,v}$ with $\psi\ge0$ and set $f(r)=\Gamma_{k,V}(r,s)\psi$. By Proposition \ref{prop:L1-positivity}, we have $f\ge0$. Since $V$ is bounded and $f\in\C([s,t];\L^2_{x,v})$, the function $Vf$ belongs to $\L^1(s,t;\L^2_{x,v})$. We regard $-Vf$ as a source term for the equation with zero potential. Applying Theorem \ref{thm:representation0T}, we obtain
    \begin{align*}
      \Gamma_{k,0}(t,s)\psi-\Gamma_{k,V}(t,s)\psi
       & =\int_s^t\Gamma_{k,0}(t,r)V(r)\Gamma_{k,V}(r,s)\psi\dd r
      \ge0.
    \end{align*}
    This proves the second inequality of \eqref{eq:potential-comparison}. The first inequality follows from  positivity and $|\Gamma_{k,V}(t,s)\psi|\le\Gamma_{k,V}(t,s)|\psi|$ for complex-valued $\psi$.

    The same argument applies to the backward operators. Taking norms in \eqref{eq:potential-comparison} and using density and duality as in Theorem \ref{thm:L1-extension}, we obtain the operator bounds for the extensions. The  $\L^1_{x,v}\to\L^\infty_{x,v}$ and $\L^2_{x,v}\to\L^\infty_{x,v}$ bounds give the kernel representations as in Theorem \ref{thm:nash-kernel}. Positivity of $\Gamma_{k,0}(t,s)-\Gamma_{k,V}(t,s)$ gives \eqref{eq:potential-kernel-comparison}. The adjoint and Chapman--Kolmogorov identities follow from the corresponding operator identities in Theorem \ref{thm:representation0T}.

    Finally, let $h\in\W^{1,\infty}(\R^{2d};\R)$ and set $\Gamma^h_{k,V}(t,s)=e^{h_t}\Gamma_{k,V}(t,s)e^{-h_s}$. Using \eqref{eq:potential-comparison}, we deduce
    \[
      |\Gamma^h_{k,V}(t,s)\psi|
      \le e^{h_t}\Gamma_{k,0}(t,s)(e^{-h_s}|\psi|)
      =\Gamma^h_{k,0}(t,s)|\psi|.
    \]
    Thus the perturbed bounds hold with the same constants. Under \eqref{eq:small-jump-support}, the rapid decay follows from \eqref{eq:potential-kernel-comparison} and Proposition \ref{prop:small-rapid}. This completes the proof.
  \end{proof}

In particular, the preceding bounds apply to $\Sigma=\Gamma_{k_{\le1},\ell}$ and $\Sigma^h=\Gamma^h_{k_{\le1},\ell}$ by comparison with $\Gamma_{k_{\le1},0}$ and $\Gamma^h_{k_{\le1},0}$, respectively.

\section{From small-jump kernels with loss potentials to the general case}

In this section we pass from estimates for the fundamental solution for equations with small-jump kernels and loss potentials to the general case as in \eqref{eq:main-equation}. We restrict to $t=1$ and $s=0$.

\subsection{Definition and integrability of the profile}

For $N>0$, define
\begin{equation}
  \phi_N(X,V)=(1+|X|+|V|)^{-N},
  \qquad (X,V)\in\R^d\times\R^d.
\end{equation}
Define the large-jump measure $\nu$ on $\R^{2d}$ by
\begin{equation} \label{eq:def-measurenu}
  \nu(B)=\int_0^1\int_{\{|\eta|>1\}}
  \mathds{1}_B(\theta\eta,\eta)|\eta|^{-d-2\beta} \dd \eta \dd \theta ,
  \qquad B\subset\R^{2d}\ \hbox{ Borel}.
\end{equation}
Since $\beta>0$,
\[
  \nu(\R^{2d})=\frac{|\mathbb S^{d-1}|}{2\beta}<\infty.
\]
For $A>0$, we set
\begin{equation}
  \label{eq:Psi}
  \Psi_{N,A}:=\sum_{m=0}^\infty\frac{A^m}{m!}\phi_N*\nu^{*m},
\end{equation}
where convolution is taken in $\R^{2d}$, and the convolution powers of $\nu$ are defined iteratively by $\nu^{*0}:=\delta_0$ and
$\nu^{*(m+1)}:=\nu^{*m}*\nu$ for $m\geq 0$.

\begin{lem}
  For every $N>2d$ and $A>0$, $\Psi_{N,A}\in \L^1(\R^{2d})\cap \L^\infty(\R^{2d})$, and
  \begin{equation}
    \label{eq:Psi-L1}
    \norm{\Psi_{N,A}}_1
    =\norm{\phi_N}_1\exp(A\nu(\R^{2d})).
  \end{equation}
\end{lem}

\begin{proof}
  Since $\nu^{*m}(\R^{2d})=\nu(\R^{2d})^m$, the measure $\sum_{m=0}^\infty \frac{A^m}{m!}\nu^{*m}$ is finite with total mass $\exp(A\nu(\R^{2d}))$.
  Since $\phi_N\in \L^1\cap \L^\infty$ for $N>2d$, convolution with this finite measure preserves both $\L^1$ and $\L^\infty$. The identity \eqref{eq:Psi-L1} follows by Tonelli's theorem and positivity.
\end{proof}

\subsection{Duhamel expansion}

We denote the large-jump gain operator acting on functions $g=g(x,v)$ by $J_{t}$  with $J_t g(x,v):=(J_{t,x}g(x,\cdot))(v)$, that is
\begin{equation}
  J_t g(x,v)=2\int_{|h|>1}
  g(x,v+h)k(t,x,v,v+h)\dd h.
\end{equation}
Substituting $\eta=-h$ and
\[
  (\Theta_\eta g)(x,v)=g(x,v-\eta),
\]
we have the pointwise estimate
\begin{equation}
  \label{eq:J-domination}
  J_t g(x,v)
  \le 2\Lambda_k\int_{|\eta|>1}
  (\Theta_\eta g)(x,v)\frac{\dd\eta}{|\eta|^{d+2\beta}}
\end{equation}
for $g\ge0$. The operator $J_t$ is bounded on $\L^1_{x,v}$, $\L^2_{x,v}$, and $\L^\infty_{x,v}$, with norm at most $\ell_1$.

\begin{lem}
  \label{lem:Duhamel}
  For every $\psi\in \L^2(\R^{2d})$,
  \begin{equation}
    \label{eq:Duhamel-series-operator}
    \Gamma(1,0)\psi
    =
    \sum_{m=0}^\infty T_m \psi
  \end{equation}
  in $\L^2(\R^{2d})$, where $T_0=\Sigma(1,0),$ and, for $m\ge1$,
  \begin{equation}
    \label{eq:Tm}
    \begin{aligned}
      T_m \psi={} &
      \int_{0<r_1<\cdots<r_m<1}
      \Sigma(1,r_m)J_{r_m}\Sigma(r_m,r_{m-1})J_{r_{m-1}}
      \cdots                                                       \\
                  & \qquad \Sigma(r_2,r_1)J_{r_1}\Sigma(r_1,0)\psi
      \dd (r_1,\ldots,r_m) .
    \end{aligned}
  \end{equation}
  The integrals are understood as Bochner integrals in $\L^2(\R^{2d})$. Moreover,
  \begin{equation}
    \label{eq:Tm-L2-bound}
    \norm{T_m \psi}_2
    \le
    \frac{\ell_1^m}{m!}\norm \psi _2,
    \qquad m\ge0,
  \end{equation}
  and the series in \eqref{eq:Duhamel-series-operator} converges absolutely in operator norm on $\L^2(\R^{2d})$.
\end{lem}

\begin{proof}
  Let $\psi\in\L^2_{x,v}$ and set $f(t)=\Gamma(t,0)\psi$. By Lemma \ref{lem:basic-properties}, we have $\cA_k+\cJ=\cA_{k_{\le1},\ell}$. Hence
  \[
    (\partial_t+v\cdot\nabla_x+\cA_{k_{\le1},\ell})f=\cJ f.
  \]
  Moreover,
  \[
    \|\Sigma(t,s)\|_{2\to2}\le1,\qquad
    \|\Gamma(t,s)\|_{2\to2}\le1,\qquad
    \|J_t\|_{2\to2}\le\ell_1.
  \]
  Since $f\in\C([0,1];\L^2_{x,v})$, the source term $\cJ f$ belongs to $\L^\infty((0,1);\L^2_{x,v})$. Applying Theorem \ref{thm:representation0T} to the small-jump equation with loss potential, we obtain since $(\cJ f)(t)= J_tf(t)$
  \begin{equation}
    \label{eq:Duhamel-one-step}
    f(t)=\Sigma(t,0)\psi+\int_0^t\Sigma(t,r)J_rf(r)\dd r,
    \qquad 0\le t\le1,
  \end{equation}
  in $\L^2_{x,v}$.

  Let us now iterate \eqref{eq:Duhamel-one-step}. Repeated substitution gives, for every $n\ge0$,
  \[
    \Gamma(1,0)\psi=\sum_{m=0}^n T_m\psi+R_{n+1}\psi.
  \]
  Here, $R_{n+1}$ is given by \eqref{eq:Tm} with $m=n+1$ and the last factor $\Sigma(r_1,0)$ replaced by $\Gamma(r_1,0)$.

  The set $\{ (r_1,\dots,r_m) \in \R^d \, ; \, 0<r_1<\cdots<r_m<1 \}$ has volume $\frac 1{m!}$. Using the operator bounds, we deduce
  \[
    \|T_m\|_{2\to2}\le\frac{\ell_1^m}{m!},\qquad m\ge0,
    \qquad
    \|R_{n+1}\|_{2\to2}\le\frac{\ell_1^{n+1}}{(n+1)!}.
  \]
  Thus $R_{n+1}\to0$ in operator norm and $\sum_{m=0}^\infty\|T_m\|_{2\to2}\le e^{\ell_1}$. This proves \eqref{eq:Duhamel-series-operator} and \eqref{eq:Tm-L2-bound}, including absolute convergence in operator norm.
\end{proof}

\subsection{Reduction to rapid small-jump decay}

The following lemma provides an analytic counterpart, in the present rough-kernel setting, of the decomposition used by Hou and Zhang  \cite[Section~1.3, equation~(1.20)]{Hou2026}.

\begin{lem}
  \label{lem:SJ}
  Let $\Sigma(t,s)=\Gamma_{k_{\le1},\ell}(t,s)$ be the small-jump with loss potential fundamental solution operator.
  For every integer $m\ge0$, every $0<r_1<\cdots<r_m<1$, and every $\eta_1,\ldots,\eta_m\in\R^d$, we define
  \begin{equation}
    \mathcal K_m^{\eta,r}:=
    \begin{cases}
      \Sigma(1,r_m)\Theta_{\eta_m}\Sigma(r_m,r_{m-1})\cdots
      \Theta_{\eta_1}\Sigma(r_1,0), & m\ge1, \\
      \Sigma(1,0),                  & m=0.
    \end{cases}
  \end{equation}
  Then $\cK_m^{\eta,r}$ is represented by a nonnegative measurable kernel $K^{\eta,r}_m(x,v,y,w)$ satisfying
  \begin{equation}
    \label{eq:SJ}
    K^{\eta,r}_m(x,v,y,w)
    \le C_N\phi_N\!\left(
    x-y-w-\sum_{j=1}^m(1-r_j)\eta_j,
    v-w-\sum_{j=1}^m\eta_j
    \right)
  \end{equation}
  for almost every $(x,v,y,w)$ and every $N>0$. The constant $C_N$ depends only on$N,d,\beta,\lambda_k,\lambda_k$ and is independent of $m,r_j,\eta_j$.
\end{lem}

\begin{proof}
  The case $m=0$ follows from Propositions~\ref{prop:small-rapid} and~\ref{prop:nonnegative-potential}. Let us now assume $m\ge1$ and fix $r_j$ and $\eta_j$ as in the statement. The key task is to remove the shifts by a change of variables.

  \emph{Step 1: The change of variables.}
  Put $r_0=0$ and $r_{m+1}=1$. For a function $b=b(r)$ with one-sided
  limits at $r_j$, we write
  \[
    b(r_j-)=\lim_{q\uparrow r_j}b(q),
    \qquad
    b(r_j+)=\lim_{q\downarrow r_j}b(q).
  \]

  Fix $\psi\in\L^2_{x,v}$. Set $u(q)=\Sigma(q,0)\psi$ for $0\le q<r_1$ and,
  iteratively for $j=1,\ldots,m$, we define
  \[
    u(r_j+)=\Theta_{\eta_j}u(r_j-),
    \qquad
    u(q)=\Sigma(q,r_j)u(r_j+),\quad r_j<q<r_{j+1}.
  \]
  The traces are well-defined by the $\L^2_{x,v}$ continuity. At the endpoint, we set
  $u(1)=\Sigma(1,r_m)u(r_m+)=\cK_m^{\eta,r}\psi$.

  We now introduce the shifted variables which can be thought of as kinetic trajectories with Dirac forcing at $r_1,\dots,r_m$; compare \cite{dietert2025criticaltrajectorieskineticgeometry}. We set
  \[
    \gamma_v(q)=\sum_{\substack{1\le i\le m\\r_i<q}}\eta_i,
    \qquad
    \gamma_x(q)=\int_0^q\gamma_v(a)\dd a
    =\sum_{\substack{1\le i\le m\\r_i<q}}(q-r_i)\eta_i,
    \quad 0\le q\le1.
  \]
  On $(r_j,r_{j+1})$, $\gamma_v$ is constant $\sum_{i=1}^j\eta_i$ and
  $ \dot{\gamma}_x=\gamma_v$. Moreover, $\gamma_x$ is continuous while $\gamma_v$ jumps at $r_j$ with
  \[
    \gamma_v(r_j+)-\gamma_v(r_j-)=\eta_j,
    \qquad j=1,\ldots,m.
  \]
  The trajectories start at $(\gamma_x(0),\gamma_v(0)) =(0,0)$ and end at
  \[
    \gamma_v(1)=\sum_{i=1}^m\eta_i,
    \qquad
    \gamma_x(1)=\sum_{i=1}^m(1-r_i)\eta_i.
  \]

  Away from the $r_j$, we define
  \[
    \widetilde u(q,x,v)
    =u(q,x+\gamma_x(q),v+\gamma_v(q)).
  \]
  We observe that taking limits and using the relation between $u(r_j+)$ and $u(r_j-)$
  \begin{align*}
    \widetilde u(r_j+,x,v)
     & =u(r_j+,x+\gamma_x(r_j),v+\gamma_v(r_j+))        \\
     & =u(r_j-,x+\gamma_x(r_j),v+\gamma_v(r_j+)-\eta_j) \\
     & =u(r_j-,x+\gamma_x(r_j),v+\gamma_v(r_j-))        \\
     & =\widetilde u(r_j-,x,v).
  \end{align*}
  Hence, we obtain that $\tilde{u}$ extends by continuity at $r_j$ so that  $\widetilde u\in\C([0,1];\L^2_{x,v})$ with $\widetilde u(0)=\psi$.

  \emph{Step 2: The equation in shifted variables.}
  We define
  \begin{align}
    \widetilde k(q,x,v,v')
     & :=k(q,x+\gamma_x(q),v+\gamma_v(q),v'+\gamma_v(q)),\notag \\
    \widetilde k_{\le1}(q,x,v,v')
     & :=\widetilde k(q,x,v,v')\mathds{1}_{\{|v-v'|\le1\}},
    \\
    \widetilde\ell(q,x,v)
     & :=\ell(q,x+\gamma_x(q),v+\gamma_v(q)).
  \end{align}
  In particular,
  \[
    \widetilde\ell(q,x,v)
    =2\int_{|h|>1}\widetilde k(q,x,v,v+h)\dd h.
  \]

  Let us next check the coefficient bounds. Both velocity arguments are shifted
  \[
    |(v+\gamma_v(q))-(v'+\gamma_v(q))|=|v-v'|.
  \]
  Symmetry is also preserved, and
  \[
    \lambda_k\frac{\mathds{1}_{\{|v-v'|\le1\}}}{|v-v'|^{d+2\beta}}
    \le\widetilde k_{\le1}(q,x,v,v')
    \le\Lambda_k\frac{\mathds{1}_{\{|v-v'|\le1\}}}{|v-v'|^{d+2\beta}}
  \]
  almost everywhere. Thus \eqref{eq:upperbound}, \eqref{eq:local-stable-lower}, and \eqref{eq:small-jump-support} hold
  with the original constants. By Lemma~\ref{lem:basic-properties}, $0\le\widetilde\ell\le\ell_1$.

  Next, we prove that $\widetilde u$ satisfies the transformed equation 
  \begin{equation} \label{eq:shifted-weak-pde}
    (\partial_q+v\cdot\nabla_x)\widetilde u+\cA_{\widetilde k_{\le1},\widetilde\ell}\ \widetilde u=0
\end{equation}
in the weak sense on $(r_j,r_{j+1})$.
  
  Fix $j\in\{0,\ldots,m\}$ and $\varphi\in\cD((r_j,r_{j+1})\times\R^{2d})$. In the unshifted variables, we use the test function
  \[
    \varphi_j(q,z,\zeta)
    =\varphi(q,z-\gamma_x(q),\zeta-\gamma_v(q)).
  \]
  This is smooth on the interval $(r_j,r_{j+1})$, since $\gamma_v$ is constant and $\gamma_x$ is affine. Using $ \dot{\gamma}_x=\gamma_v$, we obtain
  \begin{align*}
     & \left[(\partial_q+\zeta\cdot\nabla_z)\varphi_j\right]
    (q,x+\gamma_x(q),v+\gamma_v(q))                         \\
     & \qquad=\partial_q\varphi(q,x,v)
    -\gamma_v(q)\cdot\nabla_x\varphi(q,x,v)
    +(v+\gamma_v(q))\cdot\nabla_x\varphi(q,x,v)             \\
     & \qquad=(\partial_q+v\cdot\nabla_x)\varphi(q,x,v).
  \end{align*}

  In the form \eqref{eq:kernel-potential-form}, make the changes of variables
  \[
    z=x+\gamma_x(q),\qquad
    \zeta=v+\gamma_v(q),\qquad
    \zeta'=v'+\gamma_v(q),
  \]
  then
  \[
    \int_{\R^d}a^{\widetilde k_{\le1},\widetilde\ell}_{q,x}
    (\widetilde u,\varphi)\dd x
    =\int_{\R^d}a^{k_{\le1},\ell}_{q,z}(u,\varphi_j)\dd z.
  \]
  We conclude \eqref{eq:shifted-weak-pde}

  \emph{Step 3: The global weak solution.}
  Let us now glue the solutions on each subinterval to a solution on $(0,1)$. Since there are finitely many intervals and on each interval $\widetilde u$ belongs to the energy space, we have $\widetilde u\in\L^2((0,1);\L^2_x\H^\beta_v)$. By Step~1, we also have $\widetilde u\in\C([0,1];\L^2_{x,v})$.
  By uniqueness for the equation on each interval $[r_j, r_{j+1}]$, we have 
  \[
    \widetilde u(q)=\Gamma_{\widetilde k_{\le1},\widetilde\ell}\,(q,r_j)\,\psi,
    \qquad r_j\le q\le r_{j+1}.
  \] 
  If $q\in [0,1]$ and $r_j\le q\le r_{j+1}$, we thus obtain 
    \[  
        \widetilde u(q)=\Gamma_{\widetilde k_{\le1},\widetilde\ell}\,(q,r_j) \ldots \Gamma_{\widetilde k_{\le1},\widetilde\ell}\,(r_1,r_0)\,\psi
        = \Gamma_{\widetilde k_{\le1},\widetilde\ell}\,(q,0)\,\psi
    \]
    by the Chapman-Kolmogorov relation.

  \emph{Step 4: The kernel estimate.}
  We have
  \[
    (\cK_m^{\eta,r}\psi)(x,v)
    =\left[\Gamma_{\widetilde k_{\le1},\widetilde\ell}\, (1,0)\psi\right]
    (x-\gamma_x(1),v-\gamma_v(1)).
  \]
  Proposition~\ref{prop:nonnegative-potential} gives us a nonnegative measurable kernel for the operator on the right. Hence a kernel for $\cK_m^{\eta,r}$ is given by
  \[
    K^{\eta,r}_m(x,v,y,w)
    =\Gamma_{\widetilde k_{\le1},\widetilde\ell}
    \, (1,x-\gamma_x(1),v-\gamma_v(1),0,y,w).
  \]

  By Proposition~\ref{prop:nonnegative-potential} and Proposition~\ref{prop:small-rapid} we obtain
  \begin{align*}
    K^{\eta,r}_m(x,v,y,w)
     & \le\Gamma_{\widetilde k_{\le1},0}
    (1,x-\gamma_x(1),v-\gamma_v(1),0,y,w) \\
     & \le C_N\left(1+|x-\gamma_x(1)-y-w|
    +|v-\gamma_v(1)-w|\right)^{-N}
  \end{align*}
  for almost every $(x,v,y,w)$. Substituting the formulas for  $\gamma_x(1),\gamma_v(1)$ gives \eqref{eq:SJ}.
\end{proof}

\begin{prop}
  For every $N>2d$, the kernel of $\Gamma(1,0)$ satisfies
  \begin{equation}
    \label{eq:normalised-bound}
    \Gamma(1,x,v,0,y,w)
    \le C_N\Psi_{N,2\Lambda_k}(x-y-w,v-w)
  \end{equation}
  for almost every $(x,v,y,w)$. The constant $C_N$ depends only on
  $N,d,\beta,\lambda_k,\Lambda_k$.
\end{prop}

\begin{proof}
  We want to apply Lemma~\ref{lem:SJ} to each term in the Duhamel expansion. Fix $N>2d$. We denote by $T_m(x,v,y,w)$  the nonnegative kernels of the operators $T_m$ in Lemma~\ref{lem:Duhamel}. By the Duhamel expansion and positivity,
  \[
    \Gamma(1,x,v,0,y,w)
    =\sum_{m=0}^{\infty}T_m(x,v,y,w).
  \]

  Using \eqref{eq:J-domination} in each factor, Lemma~\ref{lem:SJ}, and positivity of $J_{r_i}$ and $\Sigma(r_i,r_{i+1})$, gives us
  \begin{align*}
    T_m(x,v,y,w)
     & \le (2\Lambda_k)^m
    \int_{0<r_1<\cdots<r_m<1}
    \int_{\{|\eta_1|>1,\ldots,|\eta_m|>1\}}       \\
     & \qquad
    K^{\eta,r}_m(x,v,y,w)
    \prod_{j=1}^m|\eta_j|^{-d-2\beta}\dd\eta\dd r \\
     & \le C_N(2\Lambda_k)^m
    \int_{0<r_1<\cdots<r_m<1}
    \int_{\{|\eta_1|>1,\ldots,|\eta_m|>1\}}       \\
     & \qquad
    \phi_N\!\left(
    (x-y-w,v-w)-\sum_{j=1}^m((1-r_j)\eta_j,\eta_j)
    \right)
    \prod_{j=1}^m|\eta_j|^{-d-2\beta}\dd\eta\dd r.
  \end{align*}
  The constant $C_N$ is independent of $m,r,\eta$.

  We now set $\theta_j=1-r_j$, $j=1,\ldots,m$. The integral domain transforms to
  \[
    0<\theta_m<\cdots<\theta_1<1.
  \]
  The integral in the last line is invariant under permutations of the  pairs $(\theta_j,\eta_j)$. As there are $m!$ possibilities, all with equal value, we deduce
  \begin{equation}
    \begin{aligned}
      T_m(x,v,y,w)
       & \le C_N\frac{(2\Lambda_k)^m}{m!}
      \int_{(0,1)^m}
      \int_{\{|\eta_1|>1,\ldots,|\eta_m|>1\}}           \\
       & \qquad\qquad
      \phi_N\!\left(
      (x-y-w,v-w)-\sum_{j=1}^m(\theta_j\eta_j,\eta_j)
      \right)
      \prod_{j=1}^m|\eta_j|^{-d-2\beta}\dd\eta\dd\theta \\
       & =C_N\frac{(2\Lambda_k)^m}{m!}
      (\phi_N*\nu^{*m})(x-y-w,v-w).
    \end{aligned}
  \end{equation}
  Here we used that, integration in each pair $(\theta_j,\eta_j)$ is the same as integration with respect to $\nu$ in $(\theta_j\eta_j,\eta_j)$.

  The same bound holds for $m=0$, by Lemma~\ref{lem:SJ},
  since $T_0(x,v,y,w)=\Sigma(1,x,v,0,y,w)$ and
  $\nu^{*0}=\delta_0$. Summing over $m\ge0$ and using
  \eqref{eq:Psi}, we conclude that
  \begin{align*}
    \Gamma(1,x,v,0,y,w)
     & \le C_N\sum_{m=0}^{\infty}
    \frac{(2\Lambda_k)^m}{m!}
    (\phi_N*\nu^{*m})(x-y-w,v-w)           \\
     & =C_N\Psi_{N,2\Lambda_k}(x-y-w,v-w),
  \end{align*}
  which proves \eqref{eq:normalised-bound}.
\end{proof}

\section{Comparison with the kinetic profile of order $2\beta$}

In this section, we compare the profile $\Psi_{N,A}$ defined in \eqref{eq:Psi} with the kinetic profile of order $2\beta$
\begin{equation*}
  \HZK_{d+2\beta}(x,v) = (1+|x|+|v|)^{-d-2\beta}  \int_0^1(1+|x-rv|)^{-d-2\beta}\dd r.
\end{equation*}

\begin{prop} \label{prop:cmp-Psi-HZ}
  Let $N\ge2d+4\beta$ and $A>0$. Then, for every  $(x,v) \in  \R^{2d}$,
  \begin{equation}
    \label{eq:cmp-Psi-HZ}
    \Psi_{N,A}(x,v)\le C e^{CA}\HZK_{d+2\beta}(x,v),
  \end{equation}
  where $C = C(d,\beta)>0$.
\end{prop}

Assuming the proposition is true, we can now conclude. 
\begin{proof}[Proof of Theorem~\ref{thm:main}.]
Combining \eqref{eq:cmp-Psi-HZ} and \eqref{eq:normalised-bound}, together with positivity proved in Theorem~\ref{thm:nash-kernel}, we obtain \eqref{eq:main} for $t=1,s=0$.

Let us now explain the rescaling argument to deduce it for general $s<t$. We set $\tau=t-s$. For fixed $(s,y,w)$, we introduce the new variables 
\begin{equation*}
  r=\frac{\sigma-s}{\tau},\qquad
  X=\frac{x-y-(\sigma-s)w}{\tau^{1+1/(2\beta)}},
  \qquad
  V=\frac{v-w}{\tau^{1/(2\beta)}}.
\end{equation*}
If $U(r,X,V)=u(s+\tau r,y+\tau r w+\tau^{1+1/(2\beta)} X,w+\tau^{1/(2\beta)} V)$, then
\[
  (\partial_r+V\cdot\nabla_X)U
  =\tau(\partial_t+v\cdot\nabla_x)u.
\]
The rescaled kernel is
\begin{equation*}
  \widetilde k(r,X,V,V')
  = \tau^{1+d/(2\beta)}
  k(s+\tau r,y+\tau r w+\tau^{1+1/(2\beta)} X,w+\tau^{1/(2\beta)} V,w+\tau^{1/(2\beta)} V').
\end{equation*}
Then, \eqref{eq:stable-comparable} gives
\begin{equation*}
  \lambda_k |V-V'|^{-d-2\beta}
  \le \widetilde k(r,X,V,V')
  \le \Lambda_k |V-V'|^{-d-2\beta}.
\end{equation*}
The Jacobian of $(X,V)\mapsto(x,v)$ is $\tau^{d+d/\beta}$.

Note that  if  \eqref{eq:stable-comparable} holds for $t\in (0,\infty)$, then the constant $C$ does not depend on $T$ and the estimate is valid for $0\le s<t<\infty$.
\end{proof}

The key ingredient in the proof of Proposition~\ref{prop:cmp-Psi-HZ} is the following lemma. 

\begin{lem} \label{lem:cmp-onejump}
  For every $(x,v)\in\R^{2d}$, we have
  \begin{equation}
    \label{eq:cmp-onejump}
    (\HZK_{d+2\beta}*\nu)(x,v) \le C\HZK_{d+2\beta}(x,v),
  \end{equation}
  where $C=C(d,\beta)>0$.
\end{lem}

This is essentially proven already in \cite{Hou2026}. It can be deduced from the estimates in \cite[Lemma~5]{Hou2026} and the compound Poisson
expansion in Step~1 of its proof, in particular \cite[equation~(26)]{Hou2026}, with $\alpha=2\beta$.
We provide a purely analytical proof for the convenience of the reader. 

\begin{proof}
      We start by introducing $h(u)=(1+|u|)^{-d-2\beta}$ for $u \in \R^d$. The first task is to prove that
  \begin{equation} \label{eq:cmp-scaled-convolution}
    \int_{\R^d} h(\eta) h(u-\sigma \eta) \dd\eta \le Ch(u),
    \qquad u \in \R^d, \quad \sigma  \in [-1,1].
  \end{equation}
  Since $h \in \L^1(\R^d)$ and $0< h \le1$, the cases $\sigma =0$ and $|u| \le 2$  are immediate. 
  Let us suppose now that $\sigma  \neq 0$ and $|u|>2$. When $|\sigma \eta| \le |u|/2$, we have $h(u-\sigma \eta) \le Ch(u)$, so the integral is  bounded by $C \norm{h}_1 h(u)$ in this regime. On the other hand,
  \begin{equation*}
      h(\eta)\le |\sigma |^{d+2\beta}h(u).
\end{equation*}
  Changing variables to $\zeta=u-\sigma \eta$, we obtain the estimate
  \begin{align*}
    \int_{|\sigma \eta| > |u|/2} h(\eta) h(u-\sigma \eta) \dd\eta
      &\le C |\sigma |^{d+2\beta}h(u)  \int_{\R^d} h(u-\sigma \eta) \dd \eta.   \\
      &=C|\sigma |^{2\beta}   \norm{h}_1h(u) \le C\norm{h}_1h(u).
  \end{align*}
  Together this proves \eqref{eq:cmp-scaled-convolution}.

    Let us now fix $(x,v)$ and set $\rho=1+|x|+|v|$. For $0\le \theta \le 1$  and $|\eta|>1$, we claim that
  \begin{equation}
    \frac{1}{|\eta|^{d+2\beta} (1+|x-\theta\eta|+|v-\eta|)^{d+2\beta}}
      \le\frac{C}{\rho^{d+2\beta}} (h(\eta)+h(v-\eta)).
         \label{eq:cmp-weight-split}
  \end{equation}
  First, if $|\eta|\le\rho/4$, then
  \begin{equation*}
      1+|x-\theta\eta|+ |v-\eta|   \ge \rho-2|\eta|\ge\rho/2, \quad   |\eta|^{-d-2\beta}\le2^{d+2\beta}h(\eta).
\end{equation*}
  Moreover, if $|\eta|>\rho/4$, then
  \begin{equation*}
      |\eta|^{-d-2\beta} \le 4^{d+2\beta} \rho^{-2d-2\beta},
    \quad
    (1+|x-\theta \eta|+|v-\eta |)^{-d-2\beta} \le h(v-\eta).
\end{equation*}

  By the definition of $\nu$ in \eqref{eq:def-measurenu} we have
  \begin{equation*}
      (\HZK_{d+2\beta}*\nu)(x,v)
      =\int_0^1\int_{| \eta |>1} \HZK_{d+2\beta}(x-\theta \eta ,v-\eta)  \frac{\dd \eta \dd \theta}{|\eta |^{d+2\beta}}.
\end{equation*}
  Expanding the profile, applying \eqref{eq:cmp-weight-split}, and using Tonelli's theorem, we obtain
  \begin{align}
    \label{eq:cmp-two-convolutions}
    (\HZK_{d+2\beta}*\nu)(x,v)
      &\le\frac{C}{\rho^{d+2\beta}}
        \int_0^1\int_0^1 \int_{\R^d} \left(h(\eta)+h(v-\eta)\right)\\ \nonumber
         &\quad\quad \quad\quad\quad\quad\cdot h\left(x-rv-(\theta-r)\eta\right)
        \dd\eta \dd r\dd\theta.
  \end{align}
            The first term can be bounded by \eqref{eq:cmp-scaled-convolution} as
            \begin{equation*}
                 \int_{\R^d}h(\eta) h(x-rv-(\theta-r)\eta)\dd\eta
      \le C h(x-rv),
            \end{equation*}
             since $|\theta-r| \le 1$. For the second term, we set $\zeta=v-\eta$ and apply the same estimates to obtain
  \begin{align*}
    &\int_{\R^d}h(v-\eta) h(x-rv-(\theta-r)\eta)\dd\eta\\
    &\qquad=\int_{\R^d}h(\zeta) h(x-\theta v+(\theta-r)\zeta)\dd\zeta  \le Ch(x-\theta v).
  \end{align*}
  Substituting these bounds in \eqref{eq:cmp-two-convolutions} gives
  \begin{align*}
    (\HZK_{d+2\beta}*\nu)(x,v)  &\le \frac{C}{\rho^{d+2\beta}}  \int_0^1\int_0^1  \left(h(x-rv)+h(x-\theta v)\right) \dd r \dd \theta \le C   \HZK_{d+2\beta}(x,v).
  \end{align*} 
\end{proof}

\begin{proof}[Proof of Proposition \ref{prop:cmp-Psi-HZ}]
    Since $|x-rv| \le  |x|+|v|$ for $0 \le r \le 1$, we have
  \begin{equation} \label{eq:cmp-K-lower-envelope}
    \HZK_{d+2\beta}(x,v) \ge (1+|x|+|v|)^{-2d-4\beta}.
  \end{equation}
  Thus, for $N\ge 2d+ 4 \beta$, 
  \begin{equation*}
      \phi_N(x,v) = (1+|x|+|v|)^{-N}  \le \HZK_{d+2\beta}(x,v).
\end{equation*}
Lemma   \ref{lem:cmp-onejump} and induction gives
  \begin{equation}
    \phi_N*\nu^{*m}  \le \HZK_{d+2\beta}*\nu^{*m}
      \le C^m \HZK_{d+2\beta},   \quad \quad \text{ for } m\ge0.
  \end{equation}
  Taking the series we conclude
\begin{equation*}
 	\Psi_{N,A}
      \le \HZK_{d+2\beta} \sum_{m=0}^\infty \frac{(AC)^m}{m!}
      =e^{CA}\HZK_{d+2\beta}.
\end{equation*}
\end{proof}

\appendix
\section{Stein interpolation}

We provide the argument used in Proposition~\ref{prop:perturbed-L2-Linfty}. 

\begin{lem}[Stein interpolation for kernel operators]
  \label{lem:Stein-positive-kernel}
  Let $G$ be a linear operator induced by a nonnegative measurable kernel $K(X,Y)$ on $\R^{2d}\times\R^{2d}$. Let $2\le p<\infty$ and suppose that,
  for some $B_p>0$,
  \begin{equation}
    \|G\|_{\L^p\to\L^\infty}\le B_p.
  \end{equation}
  Let $h_s,h_t$ be real bounded functions and suppose that
  \begin{equation}
    \|e^{2h_t} G e^{-2h_s}\|_{\L^p \to \L^p}\le E_p(2h),
  \end{equation}
  where the factors on the left and on the right are the multiplication operators by the corresponding functions.
  Then
  \begin{equation}
    \label{eq:Stein-p-2p}
    \|e^{h_t} G e^{-h_s}\|_{\L^p\to \L^{2p}}
    \le B_p^{1/2} E_p(2h)^{1/2}.
  \end{equation}
\end{lem}

\begin{proof}
  Positivity gives $|Gf|\le G|f|$ for complex-valued $f$, so the assumed bounds hold also for complex-valued functions.

  On the strip $0\le\Re z\le1$, consider for $f \in \L^p, g\in \L^{p'}\cap \L^1$
  \[
    (F_z f, g)=\int e^{2zh_t}Ge^{-2zh_s}f\, g .
  \]
  Since the exponential weights are bounded, the kernel representation and dominated convergence show that $z\mapsto F_z$ is an analytic family of operators on
  $0<\Re z<1$, which is continuous on $0\le\Re z\le1$ with the required bounds for Stein's interpolation theorem.
  On the boundary $\Re z=0$, multiplication by $e^{2iyh_t}$ or $e^{-2iyh_s}$ preserves absolute values. Hence
  \[
    \|F_{iy}\|_{\L^p \to \L^\infty}
    \le B_p,\qquad y\in\R.
  \]
  On the other boundary,
  \[
    F_{1+iy}
    =e^{2iyh_t}\left(e^{2h_t}Ge^{-2h_s}\right)e^{-2iyh_s},
    \qquad
    \|F_{1+iy}\|_{\L^p\to\L^p}\le E_p(2h).
  \]
  Stein interpolation at $z=1/2$, with integrability exponent $p$ for the input on both boundaries and integrability exponents $\infty$ and $p$ for the Lebesgue spaces of the output gives
  \[
    \|F_{1/2}\|_{\L^p\to\L^{2p}}
    \le B_p^{1/2} E_p(2h)^{1/2}.
  \]
  Since $F_{1/2}=e^{h_t}Ge^{-h_s}$, this is
  \eqref{eq:Stein-p-2p}.
\end{proof}

\bibliography{kke_rough}
\bibliographystyle{abbrv}

\vfill

\end{document}